\RequirePackage{fix-cm}
\documentclass[envcountsame,smallextended]{svjour3}       % onecolumn (second format)
\smartqed  % flush right qed marks, e.g. at end of proof
\usepackage{amsthm}
\usepackage{amssymb}
\usepackage{graphicx}
\usepackage[top=35mm, bottom=25mm, left=18mm, right=18mm]{geometry}
\usepackage[numbers,sort&compress]{natbib}

\usepackage{booktabs}
\usepackage{multirow}
\usepackage{orcidlink}

\usepackage{amsmath}
\usepackage{amsfonts}
\usepackage{amssymb}
\usepackage{enumitem}
\usepackage{graphicx}
\usepackage{subcaption}
\usepackage{tikz}
\usetikzlibrary{arrows.meta, positioning}
\usepackage{pgfplots}
\usetikzlibrary{patterns}
\usepackage{multicol}
\usepackage{hyperref}
\usepackage{xcolor}
\usepackage{cancel}
\usepackage{ulem}
\usepackage[active]{srcltx}

\numberwithin{equation}{section}
\usepackage{subcaption}

\usepackage{booktabs}
\usepackage{multirow}

\usepackage{algorithm}%
\usepackage{algorithmicx}%
\usepackage{algpseudocode}%

\newcommand\R{\mathbb{R}}
\newcommand\Rinf{\overline{\mathbb{R}}}
\newcommand\inter[1]{ {\rm \textbf{int}}(#1)} %interior
\newcommand\closure[1]{ {\rm \textbf{cl}}(#1)} %interior
\newcommand{\stat}{ \bs{{\rm Stat}}}
\newcommand{\crit}{ \bs{{\rm Crit}}}
\newcommand\dom[1]{ \bs{{\rm dom}}(#1)} %domain
\newcommand\Dom[1]{ \bs{{\rm Dom}}(#1)} %domain set-valued
\newcommand\dist{ \bs{{\rm dist}}} %domain set-valued
\newcommand\gf{\varphi} %general function in results
\newcommand\gh{\psi} %general function in preliminaries
\newcommand\fgam[3]{#1_{#3}^{#2}}

\newcommand\prox[3]{ \bs{{\rm prox}}_{#2#1}^{#3}}

\newcommand\ov[1]{\overline{#1}}
\newcommand\mb{\mathbf{B}}
\newcommand\bs[1]{\boldsymbol{#1}}
\newcommand\argmint[1]{\mathop{\bs{\arg\min}}\limits_{#1}}
\newcommand\Nz{\mathbb{N}_0}

\newcommand\argmin[1]{\bs{\arg\min}_{#1}}

\newcommand{\norm}[1]{\left\lVert#1\right\rVert}
\newcommand{\inner}[2]{\left\langle #1,#2\right\rangle}

\newcommand\inft[1]{\mathop{\bs{\inf}}\limits_{#1}}

\newcommand\maxt[1]{\mathop{\bs{\max}}\limits_{#1}}
\newcommand\mint[1]{\mathop{\bs{\min}}\limits_{#1}}
\let\oldliminf\liminf
\renewcommand{\liminf}{\bs{\oldliminf}}
\let\oldlimsup\limsup
\renewcommand{\limsup}{\bs{\oldlimsup}}
\let\oldmax\max
\renewcommand{\max}{\bs{\oldmax}}
\let\oldmin\min
\renewcommand{\min}{\bs{\oldmin}}
\let\oldsup\sup
\renewcommand{\sup}{\bs{\oldsup}}
\let\oldinf\inf
\renewcommand{\inf}{\bs{\oldinf}}
\let\oldlim\lim
\renewcommand{\lim}{\bs{\oldlim}}

\journalname{}

\begin{document}

\title{
Difference-of-Convex Optimization via Inexact Smoothing Descent Methods: Difference of High-Order Moreau Envelopes
}

\titlerunning{Difference-of-Convex Optimization via HOME-DC}        % if too long for running head

\author{Alireza Kabgani\,\orcidlink{0000-0001-6554-6969}
\and Moslem Zamani\,\orcidlink{0000-0003-4086-0999}
\and Masoud Ahookhosh\,\orcidlink{0000-0003-4206-9789}  %etc.
}

\authorrunning{A. Kabgani, M. Zamani, M. Ahookhosh} % if too long for running head

\institute{A. Kabgani, M. Zamani, M. Ahookhosh \at
              Department of Mathematics, University of Antwerp, Antwerp, Belgium. \\
              \email{alireza.kabgani, moslem.zamani, masoud.ahookhosh@uantwerp.be}        
            }

\date{Received: date / Accepted: date}

\maketitle

\vspace{-4mm}
\begin{abstract}
This paper studies difference-of-convex (DC) optimization problems through smoothing descent techniques. In particular, we introduce the difference of high-order Moreau envelopes (HOME-DC) and establish its fundamental and differential properties. Approximating the underlying proximal points, we generate an inexact first-order oracle for HOME-DC and characterize its accuracy guarantees. Building upon this oracle, we propose a class of inexact descent methods for minimizing DC functions and provide a convergence analysis. The proposed framework extends the applicability of envelope-based optimization techniques to a broad class of structured nonconvex problems while accommodating inexact solutions to subproblems. Preliminary numerical experiments on a sparse clustering problem demonstrate the approach's practical potential and support the theoretical findings.


% \PACS{PACS code1 \and PACS code2 \and more}
\end{abstract}

%%%%%%%%%%%%%%%%%%%%%%%%%%%%%%%%%%%%%%%%%%%%%%%%%%%%%%%%%%%%%%%%%%%%%%%%%%%%%%%%%%%%%%%%%%%%%%%%%%%%%%%%%%%%%%%%%%%%%%%%%%%%%%%%%%%%%%%%%%%%%%%%%%%%
%%% Section: Introducrtion %%%%%%%%%%%%%%%%%%%%%%%%%%%%%%%%%%%%%%%%%%%%%%%%%%%%%%%%%%%%%%%%%%%%%%%%%%%%%%%%%%%%%%%%%%%%%%%%%%%%%%%%%%%%%%%%%%%%%%%%%
\section{Introduction}
\label{intro}

Difference-of-convex (DC) optimization deals with optimization problems whose objective functions admits a decomposition as the difference of two convex functions. Specifically, we study the unconstrained DC optimization problem
\begin{equation}\label{eq:dc-problem}
    \mint{x\in \R^n}\ \gf (x) := g(x)-h(x),
\end{equation}
where the following standing assumptions are imposed unless stated otherwise.
\begin{assumption}[Basic assumptions]\label{ass:basic}
For problem \eqref{eq:dc-problem}, we assume:
\begin{enumerate}[label=(\textbf{\alph*}), font=\normalfont\bfseries, leftmargin=0.7cm]
\item \label{ass:basic:a} the functions $g,h:\R^n\to\Rinf:=\R\cup\{+\infty\}$ are proper, lower semicontinuous, convex, and possibly  nonsmooth, with $D:=\dom g\cap\dom h\neq \emptyset$;
\item\label{ass:basic:b}  the set of minimizers $\mathcal X^\star:=\argmin{x\in\R^n}\gf(x)$ is nonempty, with $\gf^\star$ denoting the optimal value.
\end{enumerate}
\end{assumption}

Problem \eqref{eq:dc-problem} encompasses a broad class of nonsmooth and nonconvex optimization models arising in areas such as sparse optimization, statistics, machine learning, clustering, signal and image processing, location problems, and global optimization; see, e.g., \cite{bagirov2018nonsmooth,bajaj2022solving,banert2019general,rotaru2025tight,Thi2018DC,LeThiPhamDinhSVM2008,LeThiPhamDinh2007Clustering,LeThiPhamDinh2015Sparse,LeThiPhamDinh2005}. Its appeal lies in the fact that the nonconvexity is entirely captured by the difference of two convex functions, thereby providing exploitable structure. Classical DC programming methods, most notably the DC Algorithm (DCA) and its boosted versions, leverage this decomposition by replacing the concave component $-h$ with affine minorants, yielding a sequence of convex subproblems; see, e.g., \cite{Abbaszadeh2024Rate,aragon2020boosted,AragonArtacho2018Accelerating,aragon2022boosted,ferreira2026inexact,TaoAn1997,HorstThoai1999,Thi2018DC}. These approaches have proven effective across a wide range of applications due to their simplicity and flexibility. However, they rely primarily on subgradient information and do not, in general, induce a smooth reformulation of the original DC objective.

Dating back to 1985, Hiriart-Urruty \cite{HiriartUrruty1985dc,hiriart1991regularize} introduced the difference of Moreau envelopes as a smoothing framework for difference-of-convex functions, namely, the function $\fgam{\gf}{}{\gamma}:\R^n\to\R$ given by
\[
    \fgam{\gf}{}{\gamma}(s):=\fgam{g}{}{\gamma}(s)-\fgam{h}{}{\gamma}(s),
\]
where $\fgam{g}{}{\gamma}$ and $\fgam{h}{}{\gamma}$ are the Moreau envelopes \cite{Moreau65} of $g$ and $h$, respectively, given by
\[
   \fgam{g}{}{\gamma}(s):=\inft{x\in\R^n}\left\{g(x)+\frac1{2\gamma}\|x-s\|^2\right\},  \quad
    \fgam{h}{}{\gamma}(s):=\inft{x\in\R^n}\left\{h(x)+\frac1{2\gamma}\|x-s\|^2\right\}.
\]
The resulting difference-of-envelopes formulation furnishes a differentiable surrogate of the original DC objective. This smoothing mechanism allows the use of first-order information and has motivated a variety of gradient-based approaches for DC optimization; see, e.g., \cite{sun2026equivalence,Sun2023DC,tang2024approximation,Themelis2020DC}. 

Recent advances in proximal mappings and Moreau envelopes with high-order regularization have attracted considerable attention due to their ability to better adapt to the geometry of the objective function; see, e.g., \cite{Nesterov2018,Ahookhosh24,Ahookhosh2025,Kabganidiff,Kabgani24itsopt,Kabgani25itsdeal,Kabgani2026Fundamental,KecisThibault15}. For a proper function $\gh$ and $p>1$, the associated high-order proximal mapping and Moreau envelope are defined by
\[
\prox{\gh}{\gamma}{p}(s) :=\argmint{y\in\R^n} \left\{\gh(y)+\frac1{p\gamma}\|s-y\|^p\right\},  \quad
 \fgam{\gh}{p}{\gamma}(s) :=\inft{y\in\R^n} \left\{\gh(y)+\frac1{p\gamma}\|s-y\|^p\right\}.
\]
These developments naturally motivate the introduction of the difference of high-order Moreau envelopes (HOME-DC)
\[
  \fgam{\gf}{p}{\gamma}(s):=\fgam{g}{p}{\gamma}(s)-\fgam{h}{p}{\gamma}(s),
\]
which serves as a smooth surrogate of the original DC objective $\varphi$. Unlike the classical quadratic case ($p=2$), however, the analysis of this envelope is considerably more delicate. Indeed, the gradient of the regularization term $\|z\|^p/p$  is given by the nonlinear duality mapping $J_p(z):=\|z\|^{p-2}z$, whose regularity properties depend fundamentally on whether  $1<p<2$, $p=2$, or $p>2$; see \cite{XuRoach,Rodomanov2020}. Consequently, the resulting HOME-DC is generally not globally Lipschitz smooth when $p\neq 2$. This observation raises new analytical and algorithmic challenges. Accordingly, this paper focuses on two complementary tasks: the analysis of the fundamental properties of HOME-DC and the development of an inexact descent framework for DC optimization problems of the form \eqref{eq:dc-problem}.

%%% Subsection: Contribution and related works %%%%%%%%%%%%%%%%%%%%%%%%%%%%%%%%%%%%%%%%%%%%%%%%%%%%%%%%%%%%%%%%%%%%%%%%%%%%%%%%%%%%%%%%%%%%%%%%%%%%%%%%%%%%%%%%%%%%%%%%
\subsection{{\bf Contribution}}
Our main contributions are summarized as follows:
\begin{description}
    \item[{\bf (i)}] {\bf Fundamental properties of HOME-DC.} We investigate several fundamental properties of HOME-DC, including two-sided approximation bounds and differential properties, all of which play a central role in the design and analysis of optimization algorithms. Particular attention is devoted to the relationship between stationary points of the original DC objective $\varphi$ and critical points of HOME-DC. In this regard, we establish sufficient conditions under which these two notions correspond, thereby providing a rigorous link between the original problem and its envelope-based reformulation.
    \item[{\bf (ii)}] {\bf Inexact descent methods via HOME-DC.} We develop a general inexact descent framework for DC optimization based on the proposed HOME-DC reformulation. An inexact first-order oracle is constructed from approximate evaluations of the underlying proximal mappings. Using this oracle, we derive an inexact sufficient decrease condition and establish a nonmonotone descent property for HOME-DC. We further prove asymptotic convergence under mild inexactness assumptions, showing that every accumulation point of the generated sequence is a critical point of HOME-DC. Finally, our preliminary numerical experiments on a sparse clustering problem illustrate the potential of the proposed framework.  
\end{description}

%%% Subsection: Organization %%%%%%%%%%%%%%%%%%%%%%%%%%%%%%%%%%%%%%%%%%%%%%%%%%%%%%%%%%%%%%%%%%%%%%%%%%%%%%%%%%%%%%%%%%%%%%%%%%%%%%%%%%%%%%%%%%%%%%%%
\subsection{{\bf Organization}}\label{sec:contribution}
The paper is organized as follows. In Section~\ref{sec:prelim}, 
we introduce the necessary notation and preliminaries.
In Section~\ref{sec:structure}, we introduce the difference of high-order Moreau envelopes (HOME-DC) and prove the structural facts:
well-definedness, optimality conditions, high-order proximal representation of DC-stationarity,
and the lift/projection relation between envelope criticality and DC-stationarity.
In Section~\ref{sec:algorithm}, we define the inexact method and establish a subsequential
convergence theorem. In Section~\ref{sec:residual}, we derive a practical residual-based inexactness condition
for the lower-level subproblems in the regime $p\ge 2$. In Section~\ref{sec:numerical}, we report our numerical experiments on a sparse clustering problem. We conclude with remarks in Section~\ref{sec:Conclusion}.

%%% End Section: Introduction %%%%%%%%%%%%%%%%%%%%%%%%%%%%%%%%%%%%%%%%%%%%%%%%%%%%%%%%%%%%%%%%%%%%%%%%%%%%%%%%%%%%%%%%%%%%%%%%%%%%%%%%%%%%%%%%%%%%%%
%%%%%%%%%%%%%%%%%%%%%%%%%%%%%%%%%%%%%%%%%%%%%%%%%%%%%%%%%%%%%%%%%%%%%%%%%%%%%%%%%%%%%%%%%%%%%%%%%%%%%%%%%%%%%%%%%%%%%%%%%%%%%%%%%%%%%%%%%%%%%%%%%%%%

%%%%%%%%%%%%%%%%%%%%%%%%%%%%%%%%%%%%%%%%%%%%%%%%%%%%%%%%%%%%%%%%%%%%%%%%%%%%%%%%%%%%%%%%%%%%%%%%%%%%%%%%%%%%%%%%%%%%%%%%%%%%%%%%%%%%%%%%%%%%%%%%%%%%
%%% Section: Preliminaries and notation %%%%%%%%%%%%%%%%%%%%%%%%%%%%%%%%%%%%%%%%%%%%%%%%%%%%%%%%%%%%%%%%%%%%%%%%%%%%%%%%%%%%%%%%%%%%%%%%%%%%%%%%%%%

\section{Preliminaries and Notation}\label{sec:prelim}

Throughout this paper, $\R^n$ denotes the $n$-dimensional \textit{Euclidean space}, 
while $\Vert \cdot \Vert$ and $\langle \cdot, \cdot \rangle$ represent the \textit{Euclidean norm}
 and \textit{inner product}, respectively. 
We denote the set of \textit{natural numbers} by $\mathbb{N}$ and let $\Nz := \mathbb{N} \cup \{0\}$. 
The set $\mb(\ov{x}; r)$ is the \textit{open ball} centered at $\ov{x} \in \R^n$ with radius $r > 0$. 
The \textit{interior} and \textit{closure} of a set $C \subseteq \R^n$ are denoted by $\inter{C}$ and
$\closure{C}$, respectively. 
For a nonempty set $C\subseteq\mathbb R^n$, the \textit{distance} from $x \in \R^n$ to $C$ is defined as
$\dist(x, C) := \bs\inf_{y \in C} \Vert y - x \Vert$. 
We adopt the convention $\infty - \infty = \infty$.
The \textit{effective domain} of $\gh: \R^n \to \Rinf := \R \cup \{+\infty\}$ is $\dom{\gh} := \{x \in \R^n \mid \gh(x) < +\infty\}$
and $\gh$ is called \textit{proper} if $\dom{\gh} \neq \emptyset$. 
The set $\mathcal{L}(\gh, \lambda) := \{x \in \R^n \mid \gh(x) \leq \lambda\}$ is the \textit{sublevel set} of $\gh$ at $\lambda \in \R$. 
The set of \textit{minimizers} of $\gh$ over $C \subseteq \R^n$ is denoted by $\argmin{x \in C} \gh(x)$. 
The function $\gh$ is \textit{lower semicontinuous} (lsc) at $\ov{x} \in \R^n$ if $\bs\liminf_{x \to \ov{x}} \gh(x) \geq \gh(\ov{x})$. 
It is \textit{coercive} if $\bs\lim_{\Vert x \Vert \to +\infty} \gh(x) = +\infty$. 
For a set-valued map $\Psi: \R^n \rightrightarrows \R^n$, the \textit{domain} is defined as $\Dom{\Psi} := \{x \in \R^n \mid \Psi(x) \neq \emptyset\}$. 
We say that a mapping $T:\R^n\to\R^n$ is \textit{locally surjective} at $(\ov s,\ov x)$, where $T(\ov s)=\ov x$, if there exist neighborhoods $U$ of $\ov s$ and $V$ of
$\ov x$ such that $V\subseteq T(U)$.

If $p > 1$, the gradient of $ \Psi_p(x):=\frac{1}{p} \Vert x \Vert^p$ is 
\[
 J_p(x):=\nabla \Psi_p(x)=\norm{x}^{p-2}x,
\]
with the convention $J_p(0)=0$.
Let $q:=\frac{p}{p-1}$ denote the conjugate exponent. Then
$J_q$ is the inverse of $J_p$, i.e.,
\[
    J_q(J_p(x))=x,     \quad    J_p(J_q(y))=y,
    \quad \forall x,y\in\R^n.
\]
Moreover, for $\alpha\ge 0$,
\[
    J_p(\alpha J_q(x))=\alpha^{p-1}x.
\]

\begin{lemma}[Basic inequalities]\label{lem:findlowbounknu:lemma} Assume that $p>1$. Then there exist constants $a_p, b_p>0$ such that, for all $x,y\in\mathbb R^n$ with $x\neq y$, the following estimates hold:
\begin{enumerate}[label=(\textbf{\alph*}), font=\normalfont\bfseries, leftmargin=0.7cm]
\item \label{lem:findlowbounknu:lemma:e1} 
$\langle J_p(x) - J_p(y), x-y\rangle\geq a_p \left(\max\{\Vert x\Vert, \Vert y\Vert\}\right)^{p-2}  \Vert x - y\Vert^2.$

\item \label{lem:findlowbounknu:lemma:e3}
$\left\Vert J_p(x) - J_p(y)\right\Vert\leq b_p(\max\{\Vert x\Vert, \Vert y\Vert\})^{p-2}\Vert x - y\Vert.$
\end{enumerate}
If $x=y$, both inequalities hold trivially after interpreting the right-hand sides as zero.
\end{lemma}
\begin{proof}
Let $x\neq y$ and set $M:=\max\{\|x\|,\|y\|\}$.
\\
\ref{lem:findlowbounknu:lemma:e1}: By \cite[Remark 1 and eq. (1.1)]{XuRoach}, there exists a constant $A_p>0$, depending only on $p$, such that
\[
\langle J_p(x) - J_p(y), x-y\rangle\geq A_p M^p \left(1-\sqrt{1-\frac{\Vert x - y\Vert^2}{16M^2}}\right).
\]
We note that
\[
1-\sqrt{1-\frac{\Vert x - y\Vert^2}{16M^2}}\geq \frac{\Vert x - y\Vert^2}{32M^2}.
\]
Thus, setting $a_p=\frac{A_p}{32}$ proves the claim.
\\
\ref{lem:findlowbounknu:lemma:e3}: By\cite[eqs. (1.3) and (3.5)]{XuRoach}, there exists a constant $B_p$, depending only on $p$, such that
\begin{equation}\label{eq1:lem:findlowbounknu:lemma:e3}
\left\Vert J_p(x) - J_p(y) \right\Vert\leq \frac{B_p M^{p}}{\Vert x - y\Vert}\left(\sqrt{1+\frac{64\Vert x - y\Vert^2}{B_p^2M^2}}-1\right).
\end{equation}
Since
\[
\sqrt{1+\frac{64\Vert x - y\Vert^2}{B_p^2M^2}}-1\leq \frac{64\Vert x - y\Vert^2}{B_p^2M^2},
\]
setting $b_p=\frac{64}{B_p}$ proves the claim.
\end{proof}

%%%%%
\begin{remark}[Regularity of the duality map]\label{rem:Jp-regularity}
When $p>2$, the factor $\left(\max\{\|x\|,\|y\|\}\right)^{p-2}$ in Lemma~\ref{lem:findlowbounknu:lemma}~\ref{lem:findlowbounknu:lemma:e3} is bounded on bounded sets, so $J_p$ is Lipschitz continuous on bounded subsets of $\mathbb R^n$. When $1<p<2$, this factor may blow up near the origin. In that range one should instead use the standard global H\"older estimate
\[
\|J_p(x)-J_p(y)\|\le 2^{2-p}\|x-y\|^{p-1}, \quad \forall x,y\in\mathbb R^n,
\]
see, for instance, \cite[Theorem~6.3]{Rodomanov2020}.
\end{remark}

A point $\ov{x}$ is called a \textit{limit point} of a sequence $\{x^k\}_{k \in \Nz}$ if $x^k \to \ov{x}$, and it is called a \textit{cluster point} if there exists a subsequence $x^j \to \ov{x}$ with $j\in J$
for some infinite subset $J \subseteq \Nz$. 

%%%%%%%%%%%%%%%%%%%%%%%%%%%%%%%%%%%%%%%%%%%%
A proper function $\gh: \R^n \to \Rinf$  is  \textit{Fr\'{e}chet differentiable} at $\ov{x}\in \inter{\dom{\gh}}$ 
with \textit{Fr\'{e}chet derivative}  
$\nabla \gh(\ov{x})$
 if 
\[
\mathop{\bs\lim}\limits_{x\to \ov{x}}\frac{\gh(x) -\gh(\ov{x}) - \langle \nabla \gh(\ov{x}) , x - \ov{x}\rangle}{\Vert x - \ov{x}\Vert}=0.
\]
For a set $C\subseteq\R^n$, the notation $\gh\in \mathcal{C}^{k}(C)$ indicates that $\gh$ is $k$-times continuously differentiable on $C$,
where $k\in \mathbb{N}$. 

\begin{definition}[H\"{o}lder continuous gradient]
A proper function $\gh: \R^n \to\Rinf$  is said to have a \textit{$\nu$-H\"{o}lder continuous gradient} on $C\subseteq \dom{\gh}$ with $\nu\in (0, 1]$ if it is Fr\'{e}chet differentiable at every point of $C$ and  there exists a constant $L_\nu\geq 0$ such that
\begin{equation}\label{eq:nu-Holder continuous gradient}
\Vert \nabla \gh(y)- \nabla \gh(x)\Vert \leq L_\nu \Vert y-x\Vert^\nu, \qquad \forall x, y\in C.
\end{equation}
\end{definition}
The class of such functions is denoted by $\mathcal{C}^{1, \nu}_{L_\nu}(C)$, and its elements are called weakly smooth.
%%%%%%%%%%%%%%%%%%%%%%%%%%%%%%%%%%%%%%%%%%%%

%%%%%%%%%%%%%%%%%%%%%%%%%%%%%%%%%%%%%%%%%%%%%%%%%%%%%%%%%%%%%%%%%%%%%%%%%%%%%%%%%%%%%%%%%%%%%%
\subsection{{\bf HOME: High-Order Moreau Envelope}} \label{sec:home}
Recent studies have renewed interest in the \textit{high-order proximal operator} (HOPE) and the \textit{high-order Moreau envelope} (HOME) as tools for designing practical algorithms for nonsmooth and nonconvex optimization %\cite{Ahookhosh2025,Kabgani24itsopt,Kabganidiff,Kabgani25itsdeal,KecisThibault15}.
\cite{Ahookhosh2025,Kabgani24itsopt,KecisThibault15}.
We briefly review these concepts and related properties.

\begin{definition}[High-order proximal operator and Moreau envelope]\label{def:hope-home}
Let $p>1$, $\gamma>0$, and $\gh: \R^n \to \Rinf$ be a proper function. 
The \textit{high-order proximal operator} (\textit{HOPE}) of $\gh$ with parameter $\gamma$,
$\prox{\gh}{\gamma}{p}: \R^n \rightrightarrows \R^n$, is 
    \begin{equation}\label{eq:Hiorder-Moreau prox}
       \prox{\gh}{\gamma}{p} (x):=\argmint{y\in \R^n} \left\{\gh(y)+\frac{1}{p\gamma}\Vert x- y\Vert^p\right\},
    \end{equation}     
and the \textit{high-order Moreau envelope} (\textit{HOME}) of $\gh$ with parameter $\gamma$, 
$\fgam{\gh}{p}{\gamma}: \R^n\to \R\cup\{\pm \infty\}$, 
is 
    \begin{equation}\label{eq:Hiorder-Moreau env}
    \fgam{\gh}{p}{\gamma}(x):=\mathop{\bs{\inf}}\limits_{y\in \R^n} \left\{\gh(y)+\frac{1}{p\gamma}\Vert x- y\Vert^p\right\}.
    \end{equation}
    \end{definition}

\begin{remark}
Definition~\ref{def:hope-home} is the direct high-order analogue of the Moreau envelope and the
proximal mapping. In the convex setting, the regularization term is coercive and strictly convex
for every $p>1$, which makes the corresponding proximal problems particularly well behaved.
In the broader nonconvex setting, the well-definedness of these high-order objects
is studied in \cite{Kabganidiff}. 
\end{remark}

\begin{fact}\cite[Proposition 12.15]{Bauschke17}\label{prop:horder:Bauschke17:p12.15}
Let $\gh:\R^n \to \Rinf$ be a proper lsc convex function and $p> 1$. Then the following hold for every $\gamma>0$ and $x\in \R^n$:
\begin{enumerate}[label=(\textbf{\alph*}), font=\normalfont\bfseries, leftmargin=0.7cm]
  \item \label{prop:horder:Bauschke17:p12.15:conv} $\fgam{\gh}{p}{\gamma}$ is convex, real-valued, continuous, and its infimum is attained;
  \item  \label{prop:horder:Bauschke17:p12.15:unique} $ \prox{\gh}{\gamma}{p}(x)$ is nonempty and single-valued.
\end{enumerate}
\end{fact}

\begin{theorem}[Lipschitz continuity of proximal operator]\label{th:home:lip}
Let $p>1$ and let $\gh:\mathbb{R}^n \to \Rinf$ be proper, lsc, and convex.
Then there exists $L_p>0$, depending only on $p$, such that for every $\gamma>0$ and all $x_1, x_2 \in \R^n$,
\begin{equation}\label{lochol:maineq}
\Vert \prox{\gh}{\gamma}{p} (x_2) -  \prox{\gh}{\gamma}{p} (x_1)\Vert\leq L_p\Vert x_2-x_1\Vert.
\end{equation}
\end{theorem}
\begin{proof} 
Let $y_i=\prox{\gh}{\gamma}{p} (x_i)$ and $v_i:=x_i-y_i$, $i=1, 2$, where $y_i$ is well-defined by Fact \ref{prop:horder:Bauschke17:p12.15}.
 The optimality condition for the high-order proximal subproblem gives
 \begin{equation}\label{lochol:eq1}
  \frac{1}{\gamma}J_p(v_i)\in \partial \gh(y_i), \quad i=1,2.
\end{equation}
By monotonicity of $\partial\gh$, $ \left\langle J_p(v_2) - J_p(v_1), y_2 -y_1\right\rangle\geq 0$.
Since $y_2 - y_1 =(x_2-x_1)-(v_2-v_1)$,
\begin{equation}\label{lochol:eq5:2}
 \langle J_p(v_2) - J_p(v_1), v_2 -v_1\rangle\leq \langle J_p(v_2) - J_p(v_1), x_2 -x_1\rangle.
\end{equation}
If $v_2 = v_1$, the claim  follows. Otherwise, by Lemma~\ref{lem:findlowbounknu:lemma}~\ref{lem:findlowbounknu:lemma:e1}, there exists a constant $a_p>0$ such that
\[
\langle J_p(v_2) - J_p(v_1), v_2-v_1\rangle\geq a_pM^{p-2}\Vert  v_2 - v_1\Vert^2,
\]
where $M= \max\{\Vert v_2\Vert, \Vert v_1\Vert\}$.
Combining this with \eqref{lochol:eq5:2} and applying the Cauchy--Schwarz inequality and Lemma~\ref{lem:findlowbounknu:lemma}~\ref{lem:findlowbounknu:lemma:e3}, we get
\[
a_pM^{p-2}\|v_2-v_1\|^2 \le \|J_p(v_2)-J_p(v_1)\|\,\|x_2-x_1\| \le b_pM^{p-2}\|v_2-v_1\|\,\|x_2-x_1\|.
\]
Since $M>0$ and $v_2\neq v_1$, cancellation gives
\[
\|y_2-y_1\| \le \|x_2-x_1\|+\|v_2-v_1\| \le \left(1+\frac{b_p}{a_p}\right)\|x_2-x_1\|.
\]
Therefore the proximal map is globally Lipschitz with $L_p:=1+\frac{b_p}{a_p}$.
\end{proof}

\begin{theorem}[Differentiability and weak smoothness of HOME]\label{th:locholof gra:th}
Let $p>1$ and let $\gh: \R^n \to \Rinf$ be proper, lsc, and convex.
Then, for each $\gamma>0$, $\fgam{\gh}{p}{\gamma}$ is continuously differentiable on $\R^n$ and
\begin{align}\label{eq:formulaofdiff}
\nabla \fgam{\gh}{p}{\gamma}(x)=\frac{1}{\gamma} J_p\bigl(x-\prox{\gh}{\gamma}{p}(x)\bigr), \quad \forall x\in\R^n.
\end{align}
Moreover, the following regularity properties hold.
\begin{enumerate}[label=(\textbf{\alph*}), font=\normalfont\bfseries, leftmargin=0.7cm]
\item\label{th:locholof gra:th:p1,2} If $p\in (1,2]$, then $\nabla\fgam{\gh}{p}{\gamma}$ is globally H\"older continuous with exponent $p-1$; in particular, it is globally Lipschitz when $p=2$. 

\item \label{th:locholof gra:th:p>2} If $p>2$, then $\nabla\fgam{\gh}{p}{\gamma}$ is Lipschitz continuous on every bounded subset of $\R^n$. 

\item \label{th:locholof gra:th:p>2:lip} If $\gh$ is Lipschitz continuous on $\R^n$, then for every $p\ge2$, $\nabla\fgam{\gh}{p}{\gamma}$ is globally Lipschitz continuous on $\mathbb R^n$.
\end{enumerate}

\end{theorem}
\begin{proof}
Differentiability of $\fgam{\gh}{p}{\gamma}$ and the formula \eqref{eq:formulaofdiff} follow from Fact \ref{prop:horder:Bauschke17:p12.15} $\ref{prop:horder:Bauschke17:p12.15:unique}$ and \cite[Proposition 20]{Kabganidiff}.
Now, we prove the second part of the theorem. For $x_1, x_2\in  \R^n$, we set
$y_i=\prox{\gh}{\gamma}{p}(x_i)$ and $v_i:=x_i-y_i$, $i=1, 2$.
\\
\ref{th:locholof gra:th:p1,2} let $p\in(1,2]$.  We have
  \begin{align*}
 \left\Vert \nabla \fgam{\gh}{p}{\gamma}(x_2)-\nabla \fgam{\gh}{p}{\gamma}(x_1)\right\Vert =\frac{1}{\gamma} \left\Vert J_p(v_2)-J_p(v_1)\right\Vert
\leq \frac{2^{2-p}}{\gamma} \left\Vert (x_2-x_1) - (y_2-y_1)\right\Vert^{p-1},
  \end{align*}
  where the last inequality comes from \cite[Theorem 6.3 ]{Rodomanov2020}. 
Combining this with Theorem~\ref{th:home:lip} ensures
  \begin{align*}
\left\Vert \nabla \fgam{\gh}{p}{\gamma}(x_2)-\nabla \fgam{\gh}{p}{\gamma}(x_1)\right\Vert  
 &\leq \frac{2^{2-p}}{\gamma}\left(\Vert x_2-x_1\Vert + \Vert y_2-y_1\Vert\right)^{p-1}\\
 &\leq  \frac{2^{2-p}}{\gamma}\left(\Vert x_2-x_1\Vert+L_p  \Vert x_2-x_1\Vert\right)^{p-1}\leq \frac{2^{2-p}}{\gamma}\left(1+L_p \right)^{p-1} \Vert x_2-x_1\Vert^{p-1},
  \end{align*}
 which proves the claim.
\\
\ref{th:locholof gra:th:p>2}
Let $p>2$. From \eqref{eq:formulaofdiff}, Lemma \ref{lem:findlowbounknu:lemma} $\ref{lem:findlowbounknu:lemma:e3}$, and Theorem \ref{th:home:lip}, with setting
\begin{equation}\label{eq:M_max_form}
M:= \left(\max\{\Vert x_2-y_2\Vert, \Vert x_1-y_1\Vert\}\right)^{p-2},
\end{equation}
there exists $b_p>0$ such that
  \begin{align*}
\left\Vert \nabla \fgam{\gh}{p}{\gamma}(x_2)-\nabla \fgam{\gh}{p}{\gamma}(x_1)\right\Vert
 =\frac{1}{\gamma} \left\Vert J_p(v_2)-J_p(v_1)\right\Vert
  &\leq \frac{b_pM}{\gamma} \left\Vert (x_2-x_1) - (y_2-y_1)\right\Vert
 \\ &\leq  \frac{b_pM}{\gamma}\left(\Vert x_2-x_1\Vert +L_p\Vert x_2-x_1\Vert\right)
 =  \frac{b_pM}{\gamma}\left(1+L_p\right)\Vert x_2-x_1\Vert.
  \end{align*}
If $x_1, x_2\in\mb(0;r)$, by Theorem~\ref{th:home:lip}, for $i=1,2$,
\[
\Vert y_i\Vert\leq \Vert \prox{\gh}{\gamma}{p} (0)\Vert+L_p\Vert x_i\Vert\leq 
\Vert\prox{\gh}{\gamma}{p} (0)\Vert+L_pr.
\]
Setting $\tau:= \Vert \prox{\gh}{\gamma}{p} (0)\Vert+L_pr$,
$\Vert x_i-y_i\Vert\leq r+\tau$ for $i=1, 2$. Thus, from $p> 2$, $M\leq \left(r+\tau\right)^{p-2}$.
Therefore,
   \begin{align*}
\left\Vert \nabla \fgam{\gh}{p}{\gamma}(x_2)-\nabla \fgam{\gh}{p}{\gamma}(x_1)\right\Vert
\leq  \frac{b_p}{\gamma}(1+L_p)\left(r+\tau\right)^{p-2}\Vert x_2-x_1\Vert.
  \end{align*}
\ref{th:locholof gra:th:p>2:lip} Suppose that $\gh$ is Lipschitz continuous on $\mathbb R^n$ with Lipschitz constant $L$. By \eqref{lochol:eq1} and \cite[Proposition 2.1.2 (a)]{Clarke90},
$\Vert x_i-y_i\Vert\leq \left(\gamma L\right)^{\frac{1}{p-1}}$ for $i=1, 2$. Thus, for $M$ given in \eqref{eq:M_max_form}, $M\leq \left(\gamma L\right)^{\frac{p-2}{p-1}}$.
Therefore, following the proof of Assertion~\ref{th:locholof gra:th:p>2}, we get
\begin{align*} 
\left\Vert \nabla \fgam{\gh}{p}{\gamma}(x_2)-\nabla \fgam{\gh}{p}{\gamma}(x_1)\right\Vert
\leq  \frac{b_p}{\gamma}\left(\gamma L\right)^{\frac{p-2}{p-1}}\left(1+L_p\right)\Vert x_2-x_1\Vert,
  \end{align*}
which proves the claim.
\end{proof}

\begin{lemma}[Optimality condition for the high-order proximal point]\label{lem:optimality}
Let $p>1$, $\gamma>0$, and let $\gh:\R^n\to\Rinf$ be proper, lsc, and convex.
Then $x=\prox{\gh}{\gamma}{p}(s)$ if and only if
$0\in \partial \gh(x)+\frac{1}{\gamma} J_p(x-s)$.
\end{lemma}

%%% End Section: Preliminaries and notation %%%%%%%%%%%%%%%%%%%%%%%%%%%%%%%%%%%%%%%%%%%%%%%%%%%%%%%%%%%%%%%%%%%%%%%%%%%%%%%%%%%%%%%%%%%%%%%%%%%%%%%
%%%%%%%%%%%%%%%%%%%%%%%%%%%%%%%%%%%%%%%%%%%%%%%%%%%%%%%%%%%%%%%%%%%%%%%%%%%%%%%%%%%%%%%%%%%%%%%%%%%%%%%%%%%%%%%%%%%%%%%%%%%%%%%%%%%%%%%%%%%%%%%%%%%%

%%%%%%%%%%%%%%%%%%%%%%%%%%%%%%%%%%%%%%%%%%%%%%%%%%%%%%%%%%%%%%%%%%%%%%%%%%%%%%%%%%%%%%%%%%%%%%%%%%%%%%%%%%%%%%%%%%%%%%%%%%%%%%%%%%%%%%%%%%%%%%%%%%%%
%%% Section: difference of high-order Moreau envelopes %%%%%%%%%%%%%%%%%%%%%%%%%%%%%%%%%%%%%%%%%%%%%%%%%%%%%%%%%%%%%%%%%%%%%%%%%%%%%%%%%%%%%%%%%%%%%%%%%
\section{HOME-DC: Difference of High-Order Moreau Envelopes }\label{sec:structure}

In this section, we introduce the \textit{difference of high-order Moreau envelopes (HOME-DC)} for a fixed DC decomposition and study its basic structural properties. These properties are used later in the design and analysis of inexact descent algorithms.

\begin{definition}[Difference of high-order Moreau envelopes (HOME-DC)]\label{def:high-order-dce}
Let $g$ and $h$ satisfy Assumption~\ref{ass:basic}~\ref{ass:basic:a}, and let $\gf=g-h$ be the associated DC decomposition on $D$. For $p>1$ and $\gamma>0$,
the \textit{difference of high-order Moreau envelopes (HOME-DC)} is the function $ \fgam{\gf}{p}{\gamma}:\R^n\to\R$ defined by
\begin{equation}\label{eq:high-order-dce}
    \fgam{\gf}{p}{\gamma}(s):=\fgam{g}{p}{\gamma}(s)-\fgam{h}{p}{\gamma}(s).
\end{equation}
By Fact~\ref{prop:horder:Bauschke17:p12.15}, both component envelopes are finite-valued. Hence, $\fgam{\gf}{p}{\gamma}$ is well-defined on $\R^n$.
\end{definition}
Although we use the notation $\fgam{\gf}{p}{\gamma}$ for the HOME-DC, it should not be confused with the high-order Moreau envelope of the DC function $\gf$ itself as defined in \eqref{eq:Hiorder-Moreau env}. In the present notation,
$\fgam{\gf}{p}{\gamma}$ denotes the difference of the two componentwise high-order Moreau envelopes associated with the chosen DC decomposition $\gf=g-h$. Thus, the construction depends on the decomposition and is generally different from applying the high-order Moreau-envelope operator directly to $\gf$.

The following proposition is an immediate consequence of Theorem~\ref{th:locholof gra:th}.
\begin{proposition}[Gradient formula and regularity of HOME-DC]
\label{prop:HOME-DC-gradient-regularity}
Let Assumption~\ref{ass:basic}~\ref{ass:basic:a} hold, $p>1$, and $\gamma>0$. Then $\fgam{\gf}{p}{\gamma}$ is continuously differentiable on $\mathbb R^n$ and
\begin{equation}\label{eq:formulaHiDCdiff}
\nabla \fgam{\gf}{p}{\gamma}(s) =  \frac{1}{\gamma} \left[ J_p\bigl(s-v(s)\bigr)-J_p\bigl(s-u(s)\bigr) \right],
\end{equation}
where $u(s):=\prox{h}{\gamma}{p}(s)$ and $v(s):=\prox{g}{\gamma}{p}(s)$.
Moreover, the following regularity properties hold.
\begin{enumerate}[label=(\textbf{\alph*}), font=\normalfont\bfseries, leftmargin=0.7cm]
\item If $p\in (1,2)$, then $\nabla\fgam{\gf}{p}{\gamma}$ is globally H\"older continuous with exponent
$p-1$.
\item If $p=2$, then $\nabla\fgam{\gf}{2}{\gamma}$ is globally Lipschitz continuous.
\item If $p>2$, then $\nabla\fgam{\gf}{p}{\gamma}$ is Lipschitz continuous on every bounded
subset of $\R^n$.
\item If both $g$ and $h$ are finite-valued and Lipschitz continuous on $\mathbb R^n$, then for every $p\ge2$, $\nabla\fgam{\gf}{p}{\gamma}$ is globally Lipschitz continuous on $\mathbb R^n$.
\end{enumerate}
\end{proposition}

%%%%%%%%%%%%%%%%%%%%%%%%%%%%%%%%%%%%%%%%%%%%%%%
The next theorem compares HOME-DC with the original DC objective at the corresponding proximal points, whenever the relevant values of $\gf$ are well-defined.
\begin{theorem}[Two-sided comparison inequality]\label{thm:two-sided-comparison}
Let Assumption~\ref{ass:basic}~\ref{ass:basic:a} hold, $p>1$, and $\gamma>0$.
Let $s\in\R^n$, and define $u=\prox{h}{\gamma}{p}(s)$ and $v=\prox{g}{\gamma}{p}(s)$.
Then  if $v\in D$, then $\gf(v)\le\fgam{\gf}{p}{\gamma}(s)$, and if $u\in D$, then $\fgam{\gf}{p}{\gamma}(s)\le\gf(u)$. In particular, if $u,v\in D$, then 
\begin{equation}\label{eq:two-sided-comparison}
    \gf(v)\le\fgam{\gf}{p}{\gamma}(s)\le
    \gf(u).
\end{equation}
\end{theorem}
\begin{proof}
By the definition of the component envelopes and the optimality of $u$ and $v$, we have
\begin{equation}\label{eq:dce-expansion}
    \fgam{\gf}{p}{\gamma}(s)=\fgam{g}{p}{\gamma}(s)-\fgam{h}{p}{\gamma}(s)=g(v)-h(u)+\frac1{p\gamma}\bigl(\norm{v-s}^p-\norm{u-s}^p\bigr).
\end{equation}
Using $g(v)-h(u)=\gf(v)+\bigl(h(v)-h(u)\bigr)$,
it follows that
\begin{equation}\label{eq:phiGammaP}
    \fgam{\gf}{p}{\gamma}(s)=\gf(v)+\bigl(h(v)-h(u)\bigr)
    +\frac1{p\gamma}\bigl(\norm{v-s}^p-\norm{u-s}^p\bigr).
\end{equation}
In addition, it follows from $h(u)+\frac1{p\gamma}\norm{u-s}^p\le h(v)+\frac1{p\gamma}\norm{v-s}^p$ that
\[
    h(v)-h(u)\ge-\frac1{p\gamma}\bigl(\norm{v-s}^p-\norm{u-s}^p\bigr).
\]
Substituting this into \eqref{eq:phiGammaP} ensures $\fgam{\gf}{p}{\gamma}(s)\ge\gf(v)$.
For the upper bound, we rewrite \eqref{eq:dce-expansion} as
\begin{equation}\label{eq:phiGammaP1}
    \fgam{\gf}{p}{\gamma}(s) = \gf(u) + \bigl(g(v)-g(u)\bigr) +
    \frac{1}{p\gamma}\bigl(\norm{v-s}^p-\norm{u-s}^p\bigr).
\end{equation}
It follows from the definition of $v$ that $g(v)+\frac1{p\gamma}\norm{v-s}^p\le  g(u)+\frac1{p\gamma}\norm{u-s}^p$, i.e.,
\[
    g(v)-g(u)\le-\frac1{p\gamma}\bigl(\norm{v-s}^p-\norm{u-s}^p\bigr).
\]
Substituting this into \eqref{eq:phiGammaP1} yields  $\fgam{\gf}{p}{\gamma}(s)\le \gf(u)$, which proves \eqref{eq:two-sided-comparison}. 
\end{proof}

\begin{remark}\label{rem:domain}
The domain qualifications in Theorem~\ref{thm:two-sided-comparison} are needed only when $g$ and $h$ are extended-valued. In that case, the lower bound should be used only when $v(s)\in D:=\dom g\cap\dom h$, and the upper bound should be used only when $u(s)\in D$. In particular,
if $ v(s)\in D$, then  $\gf(v(s))\le \fgam{\gf}{p}{\gamma}(s)$ and if $u(s)\in D$, then $\fgam{\gf}{p}{\gamma}(s)\le \gf(u(s))$.
Moreover, if $u(s)=v(s)=x\in D$, then $\fgam{\gf}{p}{\gamma}(s)=\gf(x)$. 
\end{remark}

%%%%%%%%%%%%%%%%%%%%%%%%%%%%%%%%%%%%%%%%%%%%%%%%%%%%%%%%%%%%%%%%%%%%%%%%%%%%%%%%%%%%%%%%%%%%%%%%%
\subsection{{\bf On Critical Points of HOME-DC}} \label{subsec:critical-point-type}

We next clarify the relation between critical points of HOME-DC and stationary points of the original DC objective. 

\begin{definition}[DC-stationary point]\label{def:dc-stationary}
A point $\ov x\in \dom g\cap \dom h$ is called \textit{DC-stationary} for \eqref{eq:dc-problem}
if
\begin{equation}\label{eq:dc-stationary}
    \partial g(\ov x)\cap \partial h(\ov x)\neq \emptyset.
\end{equation}
The set of DC-stationary points is denoted by $\stat(\gf)$.
\end{definition}
%%%%%%%%%%
Let us define
\[
    \crit(\fgam{\gf}{p}{\gamma}):= \{s\in\R^n:\nabla\fgam{\gf}{p}{\gamma}(s)=0\}.
\]
We now show that DC-stationarity can be expressed by the coincidence of two high-order proximal
points.
\begin{theorem}[Characterization of DC-stationarity]\label{thm:dc-stationary-representation}
Let Assumption~\ref{ass:basic}~\ref{ass:basic:a} hold and $p>1$.
Let $\ov x\in \dom g\cap\dom h$. The following are equivalent:
\begin{enumerate}[label=(\textbf{\alph*}), font=\normalfont\bfseries, leftmargin=0.7cm]
    \item\label{thm:dc-stationary-representation:a} $\ov x\in \stat(\gf)$;
    \item\label{thm:dc-stationary-representation:b} There exist $\gamma>0$ and $\ov s\in\R^n$ such that $\ov x=\prox{g}{\gamma}{p}(\ov s)=\prox{h}{\gamma}{p}(\ov s)$;
    \item\label{thm:dc-stationary-representation:c} For every $\gamma>0$, there exists $\ov s\in\R^n$ such that $\ov x=\prox{g}{\gamma}{p}(\ov s)=\prox{h}{\gamma}{p}(\ov s)$.
\end{enumerate}
\end{theorem}
\begin{proof}
\ref{thm:dc-stationary-representation:a}$\Rightarrow$ \ref{thm:dc-stationary-representation:c}:
Assume that $\ov x$ is DC-stationary. Then there exists
$\xi\in \partial g(\ov x)\cap \partial h(\ov x)$.
Fix $\gamma>0$, let $q:=p/(p-1)$, and define
\begin{equation}\label{eq:s-from-x-xi}
    \ov s:=\ov x+\gamma^{q-1}J_q(\xi).
\end{equation}
Since $J_p$ and $J_q$ are inverse maps, we obtain 
$J_p(\ov s-\ov x)=J_p(\gamma^{q-1}J_q(\xi))=(\gamma^{q-1})^{p-1}\xi=\gamma \xi$.
Equivalently, $-\frac{1}{\gamma} J_p(\ov x-\ov s)=\xi$.
Because $\xi\in \partial g(\ov x)$, Lemma~\ref{lem:optimality} implies
$\ov x=\prox{g}{\gamma}{p}(\ov s)$.
The same argument for $h$ yields $\ov x=\prox{h}{\gamma}{p}(\ov s)$.
\\
\ref{thm:dc-stationary-representation:c}$\Rightarrow$ \ref{thm:dc-stationary-representation:b}:
It is immediate.
\\
\ref{thm:dc-stationary-representation:b}$\Rightarrow$ \ref{thm:dc-stationary-representation:a}:
Assume that there exist $\gamma>0$ and $\ov s\in\R^n$ such that
$\ov x=\prox{g}{\gamma}{p}(\ov s)=\prox{h}{\gamma}{p}(\ov s)$.
Invoking Lemma~\ref{lem:optimality} implies
$-\frac{1}{\gamma} J_p(\ov x-\ov s)\in \partial g(\ov x)\cap \partial h(\ov x)$.
Hence, $\partial g(\ov x)\cap \partial h(\ov x)\neq\emptyset$, i.e., $\ov x$ is DC-stationary.
\end{proof}

\begin{theorem}[Critical points of $\fgam{\gf}{p}{\gamma}$ and DC-stationarity]\label{thm:crit-points}
Let Assumption~\ref{ass:basic}~\ref{ass:basic:a} hold, $p>1$, $\gamma>0$, and let $s\in \R^n$.
Set $u:=\prox{h}{\gamma}{p}(s)$ and $v:=\prox{g}{\gamma}{p}(s)$.
Then the following statements are equivalent:
\begin{enumerate}[label=(\textbf{\alph*}), font=\normalfont\bfseries, leftmargin=0.7cm]
    \item\label{thm:crit-points:a} $s\in \crit(\fgam{\gf}{p}{\gamma})$;
    \item\label{thm:crit-points:b} $u=v$;
\end{enumerate}
In this case, with $x:=u=v$ and $\xi:=\frac{1}{\gamma}J_p(s-x)$, one has $ \xi\in\partial g(x)\cap\partial h(x)$, $x\in\stat(\gf)$, and
$\fgam{\gf}{p}{\gamma}(s)=\gf(x)$.
\end{theorem}
\begin{proof}
By \eqref{eq:formulaHiDCdiff}, it follows that
\[
\nabla \fgam{\gf}{p}{\gamma}(s)=\nabla \fgam{g}{p}{\gamma}(s)-\nabla \fgam{h}{p}{\gamma}(s)
 =\frac{1}{\gamma}\Bigl(J_p(s-v)-J_p(s-u)\Bigr),
\]
i.e., $\nabla \fgam{\gf}{p}{\gamma}(s)=0$ if and only if
$J_p(s-v)=J_p(s-u)$. Since $J_p$ is injective for every $p>1$, this is equivalent to $u=v$.
If $x:=u=v$, then the optimality conditions for the two convex proximal subproblems give $\xi\in\partial g(x)\cap\partial h(x)$,
and hence $x\in\stat(\gf)$.  Finally, by Theorem~\ref{thm:two-sided-comparison},
$\fgam{\gf}{p}{\gamma}(s)=\gf(x)$.
\end{proof}
%%%%%%%%%
For $x\in\stat(\gf)$, by setting $q:=\frac{p}{p-1}$, we define the \textit{lifted critical set} by
\[
\mathcal L_{\gamma}^p(x):= \left\{x+\gamma^{q-1}J_q(\xi): \xi\in\partial g(x)\cap\partial h(x)\right\}.
\]

\begin{theorem}[Lifted critical set representation]\label{thm:critical-lift-description}
Let Assumption~\ref{ass:basic}~\ref{ass:basic:a} hold, $p>1$, and $\gamma>0$. Then
\[
\crit(\fgam{\gf}{p}{\gamma})= \bigcup_{x\in\stat(\gf)} \mathcal L_{\gamma}^p(x).
\]
\end{theorem}
\begin{proof}
Let $s\in\crit(\fgam{\gf}{p}{\gamma})$ and set $x:=u=v$ as introduced in Theorem~\ref{thm:crit-points}. The optimality conditions for the two proximal subproblems give
$\xi:=\frac{1}{\gamma} J_p(s-x)\in\partial g(x)\cap \partial h(x)$.
We have
\[
    s-x=J_q(\gamma\xi)=\gamma^{q-1}J_q(\xi),
\]
i.e., $s\in\mathcal L_{\gamma}^p(x)$. Conversely, take $\xi\in\partial g(x)\cap\partial h(x)$. Define $s:=x+\gamma^{q-1}J_q(\xi)$. Then
\[
    \frac{1}{\gamma} J_p(s-x)= \frac{1}{\gamma} J_p(\gamma^{q-1}J_q(\xi))=\xi \in\partial g(x)\cap\partial h(x).
\]
Equivalently,
\[
    0\in\partial g(x)+\frac{1}{\gamma} J_p(x-s),
    \quad
    0\in\partial h(x)+\frac{1}{\gamma} J_p(x-s).
\]
Hence $x=\prox{g}{\gamma}{p}(s)=\prox{h}{\gamma}{p}(s)$. Thus $\nabla\fgam{\gf}{p}{\gamma}(s)=0$.
\end{proof}

%%%%%%%%%%%%
The lift/projection mechanism described in Theorem~\ref{thm:crit-points} and Theorem~\ref{thm:critical-lift-description} is summarized schematically in Figure~\ref{fig:critical-lift}. Starting from a DC-stationary point $x$ of the original objective, each common subgradient $\xi\in\partial g(x)\cap\partial h(x)$ generates a lifted center
$s=x+\gamma^{q-1}J_q(\xi)$, which is a critical point of the HOME-DC function. In contrast, every critical point of the envelope is projected back to a
DC-stationary point of the original problem through the common proximal point $x=u(s)=v(s)$.

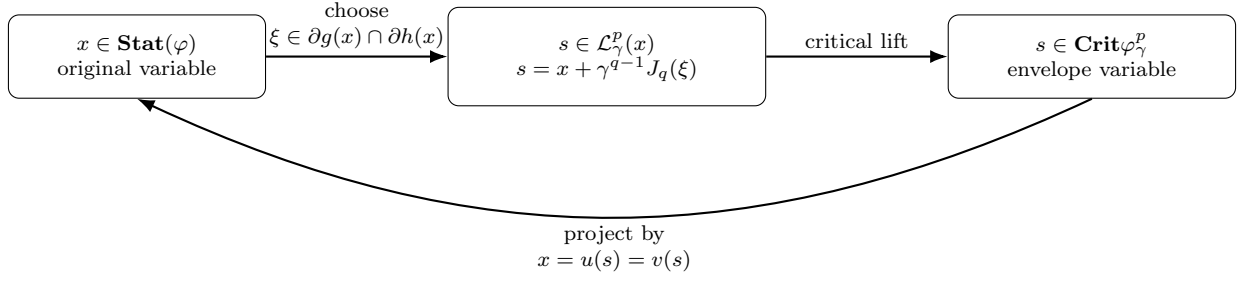
\begin{figure}[h]
\centering
\begin{tikzpicture}[node distance=2.4cm, >=Latex]
\node[draw, rounded corners, align=center, minimum width=3.4cm, minimum height=1.1cm] (stat)
{$x\in\stat(\gf)$\\ original variable};
\node[draw, rounded corners, align=center, minimum width=4.2cm, minimum height=1.3cm, right=of stat] (lift)
{$s\in\mathcal L_{\gamma}^p(x)$\\ $s=x+\gamma^{q-1}J_q(\xi)$};
\node[draw, rounded corners, align=center, minimum width=3.8cm, minimum height=1.1cm, right=of lift] (crit)
{$s\in\crit\fgam{\gf}{p}{\gamma}$\\ envelope variable};
\draw[->, thick] (stat) -- node[above, align=center] {choose\\ $\xi\in\partial g(x)\cap\partial h(x)$} (lift);
\draw[->, thick] (lift) -- node[above] {critical lift} (crit);
\draw[->, thick, bend left=25] (crit.south) to node[below, align=center] {project by\\ $x=u(s)=v(s)$} (stat.south);
\end{tikzpicture}
\caption{ Schematic lift/projection relation between DC-stationary points of the
original objective and critical points of the HOME-DC function. A stationary
point $x$ may generate several lifted critical centers $s$ when
$\partial g(x)\cap\partial h(x)$ contains more than one vector.}
\label{fig:critical-lift}
\end{figure}
%%%%%%%%%
\begin{remark}
Theorem~\ref{thm:critical-lift-description} shows that, in general, 
$\crit(\fgam{\gf}{p}{\gamma})\neq\stat(\gf)$. Rather, the two sets are related through a lift/projection mechanism. The same stationary point $x$ of the original DC objective may generate multiple, or even infinitely many, envelope critical points whenever $\partial g(x)\cap\partial h(x)$ contains more than one vector. Even when this intersection is a singleton, the lifted critical point need not satisfy $s=x$ unless
the common subgradient is zero. Example~\ref{ex:one-stationary-continuum-critical} illustrates this phenomenon.
\end{remark}

\begin{example}[One stationary point lifting to a continuum of envelope critical points]
\label{ex:one-stationary-continuum-critical}
Define $g, h:\R\to \R$ by $g(x):=|x|$ and $h(x):=2|x|$. Then $\gf(x):=-|x|$. The original DC objective has a unique DC-stationary point at $x=0$. Indeed, for $x>0$, we have $\partial g(x)=\{1\}$ and $\partial h(x)=\{2\}$ and for $x<0$, we have $\partial g(x)=\{-1\}$ and $\partial h(x)=\{-2\}$. Thus
$\partial g(x)\cap\partial h(x)=\emptyset$ whenever $x\neq 0$. At $x=0$, however, $\partial g(0)=[-1,1]$ and $\partial h(0)=[-2,2]$ and hence
$\partial g(0)\cap\partial h(0)=[-1,1]$. Therefore $\stat(\gf)=\{0\}$. 
Now, let $p>1$ and let $q=p/(p-1)$. By Theorem~\ref{thm:critical-lift-description}, the lifted critical set generated by $x=0$ is
\[
\mathcal L_{\gamma}^p(0) = \left\{ \gamma^{q-1}J_q(\xi):\xi\in[-1,1]\right\}=[-\gamma^{q-1},\gamma^{q-1}].
\]
Consequently, $\crit(\fgam{\gf}{p}{\gamma})= [-\gamma^{q-1},\gamma^{q-1}]$. Thus, a single stationary point of the original DC objective gives rise to a whole
interval of critical points of the HOME-DC function.

In the special case $p=2$ and $\gamma=1$, this interval is simply $[-1,1]$.
Moreover, the corresponding envelope $\fgam{\gf}{2}{1}$ can be written explicitly as
\[
 \fgam{\gf}{2}{1}(s)  =
    \begin{cases}
        0, & |s|\le 1,\\[2mm]
        -\dfrac12(|s|-1)^2, & 1<|s|\le 2,\\[2mm]
        \dfrac32-|s|, & |s|\ge 2.
    \end{cases}
\]
Hence $\fgam{\gf}{2}{1}$ is flat on $[-1,1]$, and every point in this interval is a critical point of the envelope, although all project to the same stationary point $x=0$ of the original DC objective. This lift from a single stationary point $x=0$ to the whole  critical  interval
$[-1,1]$ of the envelope is illustrated in
Figure~\ref{fig:one-stationary-continuum-critical}.
\end{example}

\begin{figure}[t]
\centering
\begin{tikzpicture}
\begin{axis}[
    width=0.76\textwidth,
    height=0.38\textwidth,
    axis lines=middle,
    xlabel={horizontal variable},
    ylabel={function value},
    xmin=-2.5,
    xmax=2.5,
    ymin=-1,
    ymax=0.35,
    samples=300,
    domain=-3:3,
    xtick={-3,-2,-1,0,1,2,3},
    ytick={-1.5,-1,-0.5,0},
    legend style={
        at={(0.05,0.03)},
        anchor=south west,
        draw=none,
        fill=none
    },
]

\addplot[
    very thick,
    blue,
    domain=-3:-2
]
{1.5 + x};
\addlegendentry{$\fgam{\gf}{2}{1}$}

\addplot[
    very thick,
    blue,
    domain=-2:-1,
    forget plot
]
{-0.5*(-x-1)^2};

\addplot[
    very thick,
    blue,
    domain=-1:1,
    forget plot
]
{0};

\addplot[
    very thick,
    blue,
    domain=1:2,
    forget plot
]
{-0.5*(x-1)^2};

\addplot[
    very thick,
    blue,
    domain=2:3,
    forget plot
]
{1.5 - x};

\addplot[
    very thick,
    red,
    dashed,
    domain=-3:0
]
{x};
\addlegendentry{$\gf(x)=-|x|$}

\addplot[
    very thick,
    red,
    dashed,
    domain=0:3,
    forget plot
]
{-x};

\addplot[
    ultra thick,
    black,
    forget plot
]
coordinates {(-1,0) (1,0)};

\node[black, above] at (axis cs:0,0.08)
{$\crit(\fgam{\gf}{2}{1})=[-1,1]$};

\node[red, below] at (axis cs:0,-0.55)
{$\stat(\gf)=\{0\}$};

\end{axis}
\end{tikzpicture}
\caption{A single stationary point of the original DC function may lift to a continuum of critical points of the HOME-DC function. Here
$g(x)=|x|$, $h(x)=2|x|$, $p=2$, and $\gamma=1$. The original function $\gf(x)=-|x|$ is plotted in red, while the envelope $\fgam{\gf}{2}{1}$ is plotted in blue.
The original function has the unique DC-stationary point $x=0$, whereas the envelope has the whole interval $[-1,1]$ as its critical points.}
\label{fig:one-stationary-continuum-critical}
\end{figure}
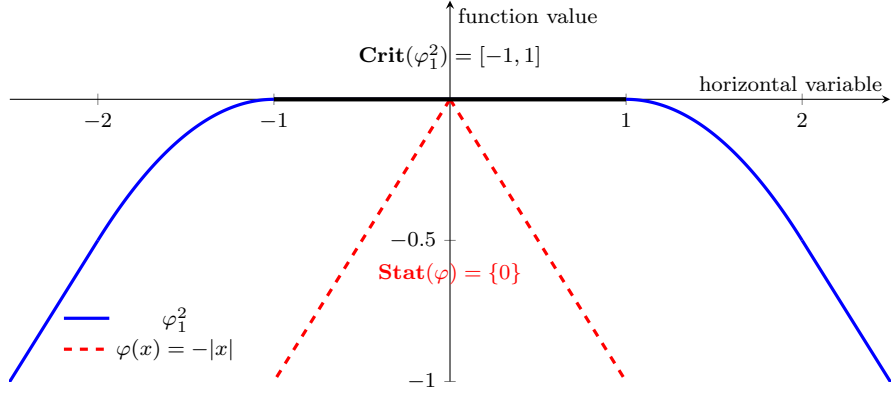

%%%%%%%%%%%
\begin{theorem}[Local minima and local maxima] \label{thm:local-extrema-type}
Let $s^\star\in\crit(\fgam{\gf}{p}{\gamma})$ and set $x^\star:=u(s^\star)=v(s^\star)$,
where $u=\prox{h}{\gamma}{p}$ and $v=\prox{g}{\gamma}{p}$. Then $\fgam{\gf}{p}{\gamma}(s^\star)=\gf(x^\star)$. Moreover, the following statements hold:
\begin{enumerate}[label=(\textbf{\alph*}), font=\normalfont\bfseries, leftmargin=0.7cm]
\item\label{thm:local-extrema-type:a} If $x^\star$ is a local minimizer (resp. maximizer) of $\gf$, then $s^\star$ is a local minimizer (resp. maximizer) of $\fgam{\gf}{p}{\gamma}$.

% \item \label{thm:local-extrema-type:b} If $x^\star$ is a local maximizer of $\gf$, then $s^\star$ is a local maximizer of $\fgam{\gf}{p}{\gamma}$.

\item \label{thm:local-extrema-type:c}
If $s^\star$ is a local minimizer (resp. maximizer) of $\fgam{\gf}{p}{\gamma}$ and $u=\prox{h}{\gamma}{p}$ (resp. $v=\prox{g}{\gamma}{p}$) is locally surjective at $(s^\star,x^\star)$, then $x^\star$ is a local minimizer (resp. maximizer) of $\gf$.

\end{enumerate}
\end{theorem}
\begin{proof}
By Theorem~\ref{thm:two-sided-comparison}, $\gf(v(s^\star))\le \fgam{\gf}{p}{\gamma}(s^\star)\le \gf(u(s^\star))$.
Since $s^\star$ is critical, $u(s^\star)=v(s^\star)=x^\star$. Therefore,
\[
    \fgam{\gf}{p}{\gamma}(s^\star)=\gf(x^\star).
\]
\ref{thm:local-extrema-type:a} 
If $x^\star$ is a local minimizer of $\gf$, then for $s$ sufficiently close to $s^\star$, the continuity of $v$ gives $v(s)$ close to $x^\star$, and hence by Theorem~\ref{thm:two-sided-comparison},
\[
    \fgam{\gf}{p}{\gamma}(s^\star) = \gf(x^\star)\le \gf(v(s))\le \fgam{\gf}{p}{\gamma}(s),
\]
i.e., $s^\star$ is a local minimizer of $\fgam{\gf}{p}{\gamma}$.
The proof for local maximizers is analogous.
\\
\ref{thm:local-extrema-type:c} 
Because $s^\star$ is a local minimizer of $\fgam{\gf}{p}{\gamma}$, there exists a neighborhood $U$ of $s^\star$ such that
\[
\fgam{\gf}{p}{\gamma}(s^\star)\le \fgam{\gf}{p}{\gamma}(s)\quad \forall s\in U.
\]
By local surjectivity of $u$, there exists a neighborhood $V$ of $x^\star$ such that $V\subseteq \prox{h}{\gamma}{p}(U)$. For $x\in V$, there exists $s\in U$ such that $x=u(s)$.
Using Theorem~\ref{thm:two-sided-comparison},
\[
\fgam{\gf}{p}{\gamma}(s)\le \gf(u(s))=\gf(x).
\]
Consequently, $\gf(x^\star)=\fgam{\gf}{p}{\gamma}(s^\star)\le \fgam{\gf}{p}{\gamma}(s)\le \gf(x)$, i.e., $\gf(x^\star)\le \gf(x)$ for all $x\in V$, which shows that $x^\star$ is a local minimizer of $\gf$. The proof for local maximizers is analogous.
\end{proof}

%%%%%%%%%%%
\begin{corollary}[Saddle points under local surjectivity]\label{cor:saddle-correspondence}
Let $s^\star\in\crit(\fgam{\gf}{p}{\gamma})$ and set $x^\star:=u(s^\star)=v(s^\star)$.
Assume that both $u=\prox{h}{\gamma}{p}$ and $v=\prox{g}{\gamma}{p}$ are locally surjective at $(s^\star,x^\star)$. If $x^\star$ is a DC-stationary point of $\gf$ that is neither a local minimizer nor a local maximizer, then $s^\star$ is a critical point of $\fgam{\gf}{p}{\gamma}$ that is neither a local minimizer nor a local maximizer.
\end{corollary}
\begin{proof}
If $s^\star$ were a local minimizer of $\fgam{\gf}{p}{\gamma}$, then Theorem~\ref{thm:local-extrema-type} would imply that $x^\star$ is a local minimizer of $\gf$, a contradiction. If $s^\star$ were a local maximizer of $\fgam{\gf}{p}{\gamma}$, then Theorem~\ref{thm:local-extrema-type} would imply that $x^\star$ is a local
maximizer of $\gf$, again a contradiction. Hence $s^\star$ is neither a local minimizer nor a local maximizer.
\end{proof}

%%%%%%
The preceding statements are summarized in Table~\ref{tab:critical-type-summary}. The table emphasizes that the correspondence is generally not an identity between two subsets of the same space. The envelope variable is the center $s$, whereas the original DC variable is the primal point $x=u(s)=v(s)$.

\begin{table}[h]
\centering
\caption{Correspondence between point types of the DC objective and the HOME-DC function.}
\label{tab:critical-type-summary}
\begin{tabular}{p{0.30\textwidth}p{0.35\textwidth}p{0.25\textwidth}}
\toprule
Object for $\gf=g-h$ & Corresponding object for $\fgam{\gf}{p}{\gamma}$ & Extra condition \\
\midrule
$x\in\stat(\gf)$ & every $s\in\mathcal L_{\gamma}^p(x)$ is critical for $\fgam{\gf}{p}{\gamma}$ & none \\
$x=u(s)=v(s)\in\stat(\gf)$ & $s\in\crit(\fgam{\gf}{p}{\gamma})$ & none \\
local minimizer $x$ & lifted critical point $s$ is a local minimizer & none \\
projected point $x=u(s)$ is a local minimizer & local minimizer $s$ & local surjectivity of $u$ \\
local maximizer $x$ & lifted critical point $s$ is a local maximizer & none \\
projected point $x=v(s)$ is a local maximizer & local maximizer $s$ & local surjectivity of $v$ \\
saddle-type stationary point $x$ & lifted critical point $s$ is saddle-type & local surjectivity of $u$ and $v$ \\
\bottomrule
\end{tabular}
\end{table}
%%%%%%%%%%
\begin{corollary}[A bounded-subgradient sufficient condition]
\label{cor:bounded-subgradient-local-min}
Let $s^\star$ be a local minimizer of $\fgam{\gf}{p}{\gamma}$. Suppose that there exists $r>0$ such that
\[
    \fgam{\gf}{p}{\gamma}(s^\star)\le \fgam{\gf}{p}{\gamma}(s),
    \quad \forall s\in\mb(s^\star;r).
\]
Set $x^\star:=u(s^\star)=v(s^\star)$ and assume that
\[
    \sup\{\|\xi\|:\xi\in\partial h(x),\ x\in\dom h\}\le M_h<+\infty.
\]
If $r>2\gamma^{q-1}M_h^{q-1}$, where $q:=p/(p-1)$, then $\gf(x^\star)\le \gf(x)$ for every $x$ satisfying
\[
    \|x-x^\star\|\le r-2\gamma^{q-1}M_h^{q-1}.
\]
In particular, $x^\star$ is a local minimizer of $\gf$.
\end{corollary}
\begin{proof}
Since $s^\star$ is a local minimizer of $\fgam{\gf}{p}{\gamma}$, Theorem~\ref{thm:crit-points} gives $u(s^\star)=v(s^\star)=x^\star$ and 
$\fgam{\gf}{p}{\gamma}(s^\star)=\gf(x^\star)$. Let
\[
    \xi^\star:=-\frac{1}{\gamma} J_p(x^\star-s^\star)\in\partial h(x^\star),
\]
so that $s^\star=x^\star+\gamma^{q-1}J_q(\xi^\star)$. Take any admissible  $x\in\R^n$ with $\partial h(x)\neq\emptyset$ and choose $\xi_x\in\partial h(x)$. Define $s_x:=x+\gamma^{q-1}J_q(\xi_x)$.
Then, by Lemma~\ref{lem:optimality}, $u(s_x)=x$. Moreover,
\[
\|s_x-s^\star\|\le \|x-x^\star\|+\gamma^{q-1}\|J_q(\xi_x)-J_q(\xi^\star)\|.
\]
Since $\|\xi_x\|\le M_h$ and $\|\xi^\star\|\le M_h$,
\[
\|J_q(\xi_x)-J_q(\xi^\star)\| \le \|J_q(\xi_x)\|+\|J_q(\xi^\star)\| \le 2M_h^{q-1}.
\]
Thus
\[
    \|s_x-s^\star\| \le  \|x-x^\star\|+2\gamma^{q-1}M_h^{q-1}.
\]
If $\|x-x^\star\|\le r-2\gamma^{q-1}M_h^{q-1}$, then $s_x\in\mb(s^\star;r)$. Hence
\[
    \gf(x^\star)=\fgam{\gf}{p}{\gamma}(s^\star) \le \fgam{\gf}{p}{\gamma}(s_x) \le \gf(u(s_x))=\gf(x),
\]
which proves the claim.
\end{proof}
%%%%%%%%
\begin{corollary}[A smooth sufficient condition for local surjectivity]\label{cor:smooth-h-local-surjectivity}
Let $s^\star\in\crit\fgam{\gf}{p}{\gamma}$ and set $x^\star:=u(s^\star)=v(s^\star)$,
where $u=\prox{h}{\gamma}{p}$ and $v=\prox{g}{\gamma}{p}$.  Suppose additionally that $h$ is continuously differentiable in a neighborhood of
$x^\star$. Then $u=\prox{h}{\gamma}{p}$ is locally surjective at $(s^\star,x^\star)$.
\end{corollary}
\begin{proof}
Since $h$ is continuously differentiable in a neighborhood of $x^\star$, define, for $x$ in this neighborhood,
\[
    S_h(x):=x+\gamma^{q-1}J_q(\nabla h(x)).
\]
The map $S_h$ is continuous near $x^\star$. Moreover, for every $x$ sufficiently close to $x^\star$, the high-order proximal optimality condition gives
$x=\prox{h}{\gamma}{p}(S_h(x))$.
Since $S_h(x)\to S_h(x^\star)=s^\star$ as $x\to x^\star$, every point $x$ sufficiently
close to $x^\star$ has a preimage $S_h(x)$ arbitrarily close to $s^\star$. Hence
$\prox{h}{\gamma}{p}$ is locally surjective at $(s^\star,x^\star)$.
\end{proof}
%%%%%%%%%
The following elementary examples illustrate the phenomena in Theorems~\ref{thm:crit-points},~\ref{thm:critical-lift-description}, and~\ref{thm:local-extrema-type}. They
illustrate which local point properties are preserved by the envelope and which properties depend on the lifting map.
%%%%%%%%%%%%
\begin{example}[A minimizer need not occur at the same point]
\label{ex:minimizer-shift}
Let $g, h: \R\to\R$ be defined by $g(x):=x^2+x$ and $h(x):=x$. Then $\gf(x)=g(x)-h(x)=x^2$. Hence, the original DC objective has a unique global minimizer at $x^\star=0$. Moreover, $\partial g(0)=\partial h(0)=\{1\}$. Thus the lifted HOME-DC critical point is
\[
    s^\star=0+\gamma^{q-1}J_q(1)=\gamma^{q-1}.
\]
Therefore the corresponding envelope minimizer is located at $s^\star=\gamma^{q-1}$, whereas the minimizer of the original problem is located at $x^\star=0$. For $p=2$,
the center is shifted from $x^\star=0$ to $s^\star=\gamma$. 
\end{example}
Similar behavior may occur for maximizers and saddle-type points.
%%%%%%%%%
\begin{example}[The lift depends on the chosen DC decomposition]
\label{ex:decomposition-dependent-lift}
For every $c\in\R$, the same function $\gf(x)=x^2$ admits the DC representation
\[
    \gf(x)=g_c(x)-h_c(x),  \quad g_c(x):=x^2+cx, \quad  h_c(x):=cx.
\]
The original objective has the same unique minimizer $x^\star=0$ for every $c$.
However,
\[
    \partial g_c(0)\cap\partial h_c(0)=\{c\},
\]
and hence the lifted envelope critical point is
\[
    s_c^\star =\gamma^{q-1}J_q(c) = \gamma^{q-1}|c|^{q-2}c.
\]
Thus different DC decompositions of the same function may produce different lifted critical points of the envelope. This is not a contradiction, since the HOME-DC function depends on the ordered pair $(g,h)$, not merely on the difference $g-h$.
\end{example}

\begin{theorem}[Minimizers of $\gf$ and $\fgam{\gf}{p}{\gamma}$]\label{thm:minimizer-correspondence}
Let $S_\star:=\argmint{s\in \R^n}\fgam{\gf}{p}{\gamma}(s)$.
Then
\begin{equation}\label{eq:min-correspondence-p}
    \argmint{x\in\R^n} \gf(x)=\prox{h}{\gamma}{p}(S_\star)=\prox{g}{\gamma}{p}(S_\star),
\end{equation}
and
\begin{equation}\label{eq:inf-correspondence-p}
   {\mathop {\mathrm{\inf}}\limits_{x\in\R^n}}\gf(x) = {\mathop {\mathrm{\inf}}\limits_{s\in \R^n}} \fgam{\gf}{p}{\gamma}(s).
\end{equation}
\end{theorem}
\begin{proof}
    Let $s\in\R^n$ and define $u:=\prox{h}{\gamma}{p}(s)$ and $v:=\prox{g}{\gamma}{p}(s)$.
By Theorem~\ref{thm:two-sided-comparison}, $\gf(v)\le \fgam{\gf}{p}{\gamma}(s)$, and hence
\[
    {\mathop {\mathrm{\inf}}\limits_{x\in\R^n}}\gf(x) \le \fgam{\gf}{p}{\gamma}(s)
    \quad \forall s\in\R^n.
\]
Taking the infimum over $s$ yields
\begin{equation}\label{eq:inf-ineq-1}
{\mathop {\mathrm{\inf}}\limits_{x\in\R^n}}\gf(x) \leq {\mathop {\mathrm{\inf}}\limits_{s\in \R^n}} \fgam{\gf}{p}{\gamma}(s).
\end{equation}
For $x^\star\in\argmin{x\in\R^n}\gf$ and for all $y\in\R^n$, we have $g(y)-h(y)\ge g(x^\star)-h(x^\star)$, i.e., $g(y)\ge g(x^\star)+h(y)-h(x^\star)$.
Let $\xi\in \partial h(x^\star)$. Since $h$ is convex,
\[
    h(y)\ge h(x^\star)+\langle \xi,y-x^\star\rangle,
    \quad \forall y\in\R^n.
\]
Substituting this inequality into the preceding display yields
\[
    g(y)\ge g(x^\star)+\langle \xi,y-x^\star\rangle,
    \quad \forall y\in\R^n,
\]
i.e., $\xi\in \partial g(x^\star)$, which implies $\partial h(x^\star)\subseteq \partial g(x^\star)$.
Hence, $x^\star$ is DC-stationary. By Theorem~\ref{thm:dc-stationary-representation}, there exists $s^\star\in\R^n$ such that
\[
    x^\star=\prox{g}{\gamma}{p}(s^\star)=\prox{h}{\gamma}{p}(s^\star).
\]
Applying Theorem~\ref{thm:two-sided-comparison} with $u=v=x^\star$, we obtain
\[
    \fgam{\gf}{p}{\gamma}(s^\star)=\gf(x^\star)={\mathop {\mathrm{\inf}}\limits_{x\in\R^n}}\gf(x),
\]
i.e., $\inf_{s\in \R^n} \fgam{\gf}{p}{\gamma}(s)\le \inf_{x\in \R^n}\gf(x)$. Combined with \eqref{eq:inf-ineq-1}, we get \eqref{eq:inf-correspondence-p}.
Moreover, $s^\star\in S_\star$, i.e.,
\[
    x^\star\in \prox{h}{\gamma}{p}(S_\star)\cap \prox{g}{\gamma}{p}(S_\star),
\]
leading to
\begin{equation}\label{eq:subset-forward-p}
    \argmint{x\in\R^n}\gf(x)\subseteq\prox{h}{\gamma}{p}(S_\star),
    \quad
    \argmint{x\in\R^n}\gf(x)\subseteq\prox{g}{\gamma}{p}(S_\star).
\end{equation}
Let $s^\star\in S_\star$. Since $\fgam{\gf}{p}{\gamma}$ is differentiable and $s^\star$ minimizes it,
$\nabla \fgam{\gf}{p}{\gamma}(s^\star)=0$, i.e.,
\[
    \prox{g}{\gamma}{p}(s^\star)=\prox{h}{\gamma}{p}(s^\star).
\]
Define $x^\star:=\prox{h}{\gamma}{p}(s^\star)=\prox{g}{\gamma}{p}(s^\star)$. Applying Theorem~\ref{thm:two-sided-comparison}, it follows that $\fgam{\gf}{p}{\gamma}(s^\star)=\gf(x^\star)$.
Since $s^\star\in S_\star$,
\[
    \gf(x^\star)=\fgam{\gf}{p}{\gamma}(s^\star)={\mathop {\mathrm{\inf}}\limits_{s\in \R^n}} \fgam{\gf}{p}{\gamma}(s)={\mathop {\mathrm{\inf}}\limits_{x\in\R^n}}\gf(x),
\]
i.e., $x^\star\in\argmin{x\in\R^n}\gf(x)$, leading to
\[
    \prox{h}{\gamma}{p}(S_\star)\subseteq \argmint{x\in\R^n}\gf(x),
    \quad
    \prox{g}{\gamma}{p}(S_\star)\subseteq \argmint{x\in\R^n}\gf(x).
\]
Together with \eqref{eq:subset-forward-p}, this ensures \eqref{eq:min-correspondence-p}.
\end{proof}

%%% End Section: difference of high-order Moreau envelopes %%%%%%%%%%%%%%%%%%%%%%%%%%%%%%%%%%%%%%%%%%%%%%%%%%%%%%%%%%%%%%%%%%%%%%%%%%
%%%%%%%%%%%%%%%%%%%%%%%%%%%%%%%%%%%%%%%%%%%%%%%%%%%%%%%%%%%%%%%%%%%%%%%%%%%%%%%%%%%%%%%%%%%%%%%%%%%%%%%%%%%%%%%%%%%%%%%%%%%%%%%%%%%%%%%%%%%%%%%%%%%%

%%%%%%%%%%%%%%%%%%%%%%%%%%%%%%%%%%%%%%%%%%%%%%%%%%%%%%%%%%%%%%%%%%%%%%%%%%%%%%%%%%%%%%%%%%%%%%%%%%%%%%%%%%%%%%%%%%%%%%%%%%%%%%%%%%%%%%%%%%%%%%%%%%%%
%%% Section: Inexact high-order two-prox methods %%%%%%%%%%%%%%%%%%%%%%%%%%%%%%%%%%%%%%%%%%%%%%%%%%%%%%%%%%%%%%%%%%%%%%%%%%%%%%%%%%%%%%%%%%%%%%%%%
\section{IDEA: Inexact Descent Algorithms via HOME-DC}\label{sec:algorithm}
In this section, we study an inexact descent framework applied to HOME-DC function. We formulate the method in terms of a general gradient-related
search direction computed from an approximation of the envelope gradient. This includes the gradient method as a special case and also covers safeguarded quasi-Newton-type
directions, including L-BFGS-type directions, whenever the stated descent and boundedness conditions hold.
%%%%%%%%%%%%%%%%%%%%%%%%%%%%%%%%%%%%%%%%%%%%%%%%%%%%%%%%%%%%%%%%%%%%%%%%%%%%%%%%%%%%%%%%%%%%%%%%%
\subsection{{\bf Inexact Oracle for HOME-DC}}\label{sec:inexactOracle}

In this section, we discuss the construction of an inexact first-order oracle for HOME-DC function. Let $\widetilde u^k,\widetilde v^k$ be approximate solutions of the two high-order proximal subproblems at iteration $k$:
\[
    \widetilde u^k \approx \prox{h}{\gamma}{p}(s^k),
    \quad
    \widetilde v^k \approx \prox{g}{\gamma}{p}(s^k).
\]
The inexact first-order oracle is defined by
\begin{equation}\label{eq:approx-gradient}
    \widetilde G_k:=\frac{1}{\gamma}\Bigl(J_p(s^k-\widetilde v^k)-J_p(s^k-\widetilde u^k)\Bigr).
\end{equation}
The corresponding oracle error is
\[
    e^k:=\widetilde G_k-\nabla\fgam{\gf}{p}{\gamma}(s^k).
\]
The role of the lower-level inexactness analysis is to impose conditions on $\widetilde u^k$ and $\widetilde v^k$ ensuring that the accumulated oracle error $\{e^k\}$ is small enough.  Section~\ref{sec:residual} later gives residual-based sufficient conditions for this requirement.

\begin{theorem}[Asymptotic recovery of primal minimizers from approximate proximal points]\label{thm:inexact-primal-recovery}
Let $S_\star:=\argmint{s\in \R^n}\fgam{\gf}{p}{\gamma}(s)$.
Suppose $\{s^k\}_{k\in\Nz}$  satisfies
\begin{equation}\label{eq:s-to-Sstar}
    \dist(s^k,S_\star)\to 0.
\end{equation}
For each $k\in \Nz$, define the exact proximal points
$u^k:=\prox{h}{\gamma}{p}(s^k)$ and $v^k:=\prox{g}{\gamma}{p}(s^k)$
and let $\widetilde u^k,\widetilde v^k\in\R^n$ be approximate proximal points satisfying
\begin{equation}\label{eq:prox-approx-vanish}
    \norm{\widetilde u^k-u^k}\to 0,
    \quad
    \norm{\widetilde v^k-v^k}\to 0.
\end{equation}
Then
\begin{equation}\label{eq:dist-u-min}
    \dist(u^k,\mathcal X^\star)\to 0,
    \quad
    \dist(v^k,\mathcal X^\star)\to 0,
\end{equation}
and
\begin{equation}\label{eq:dist-utilde-min}
    \dist(\widetilde u^k,\mathcal X^\star)\to 0,
    \quad
    \dist(\widetilde v^k,\mathcal X^\star)\to 0.
\end{equation}
\end{theorem}

\begin{proof}
We first prove \eqref{eq:dist-u-min}.
Assume, for contradiction, that $\dist(u^k,\mathcal X^\star)\not\to 0$. Then there exist
$\varepsilon>0$ and a subsequence $\{u^{k_j}\}_{j\in\Nz}$ such that
\[
    \dist(u^{k_j},\mathcal X^\star)\ge \varepsilon, \qquad \forall j\in \Nz.
\]
Since $\dist(s^k,S_\star)\to 0$, for each $j$ we can choose $\hat s^{k_j}\in S_\star$  satisfying
\[
    \norm{s^{k_j}-\hat s^{k_j}}\le \dist(s^{k_j},S_\star)+\frac1j \to 0.
\]
Set $\hat u^{k_j}:=\prox{h}{\gamma}{p}(\hat s^{k_j})$.
By Theorem~\ref{thm:minimizer-correspondence},
$\hat u^{k_j}\in \mathcal X^\star$ for all $j\in\Nz$.
Since $h$ is proper, lsc, and convex, the map $\prox{h}{\gamma}{p}$ is continuous.
Hence
\[
    u^{k_j} =\prox{h}{\gamma}{p}(s^{k_j})
    \to\prox{h}{\gamma}{p}(\hat s^{k_j}) = \hat u^{k_j}.
\]
More precisely $\norm{u^{k_j}-\hat u^{k_j}}\to 0$, and thus
\[
    \dist(u^{k_j},\mathcal X^\star)\le\norm{u^{k_j}-\hat u^{k_j}}\to 0,
\]
which contradicts $\dist(u^{k_j},\mathcal X^\star)\ge\varepsilon$.
Thus $\dist(u^k,\mathcal X^\star)\to 0$.
The proof for $v^k$ is identical, using the continuity of $\prox{g}{\gamma}{p}$ and the fact that
$\prox{g}{\gamma}{p}(S_\star)=\mathcal X^\star$, 
by Theorem~\ref{thm:minimizer-correspondence}.

We now prove \eqref{eq:dist-utilde-min}. By the triangle inequality,
\[
    \dist(\widetilde u^k,\mathcal X^\star)
    \le
    \norm{\widetilde u^k-u^k}+\dist(u^k,\mathcal X^\star).
\]
Using \eqref{eq:prox-approx-vanish} and \eqref{eq:dist-u-min}, the right-hand side tends to zero.
Hence $\dist(\widetilde u^k,\mathcal X^\star)\to 0$.
The proof for $\widetilde v^k$ is the same.
\end{proof}

\begin{corollary}[Pointwise convergence of the approximate proximal points]\label{cor:pointwise-primal-recovery}
Let the assumptions of Theorem~\ref{thm:inexact-primal-recovery} hold, and assume in addition that
$s^k\to s^\star$ for some $s^\star\in S_\star$.
Define $x^\star:=\prox{h}{\gamma}{p}(s^\star)=\prox{g}{\gamma}{p}(s^\star)$.
Then $x^\star\in\mathcal X^\star$, and
$u^k\to x^\star$ and $v^k\to x^\star$.
If moreover
\[
    \norm{\widetilde u^k-u^k}\to 0,
    \quad
    \norm{\widetilde v^k-v^k}\to 0,
\]
then
\[
    \widetilde u^k\to x^\star,
    \quad
    \widetilde v^k\to x^\star.
\]
\end{corollary}

\begin{proof}
Since $s^\star\in S_\star$, Theorem~\ref{thm:minimizer-correspondence} implies
$x^\star=\prox{h}{\gamma}{p}(s^\star)=\prox{g}{\gamma}{p}(s^\star)\in\argmin{}\gf$.
By continuity of the high-order proximal mappings,
\[
    u^k=\prox{h}{\gamma}{p}(s^k)\to \prox{h}{\gamma}{p}(s^\star)=x^\star,
\]
and similarly
\[
    v^k=\prox{g}{\gamma}{p}(s^k)\to \prox{g}{\gamma}{p}(s^\star)=x^\star.
\]
The convergence of $\widetilde u^k$ and $\widetilde v^k$ follows from the triangle inequality.
\end{proof}

%%%%%%%%%%%%%%%%%%%%%%%%%%%%%%%%%%%%%%%%%%%%%%%%%%%%%%%%%%%%%%%%%%%%%%%%%%%%%%%%%%%%%%%%%%%%%%%%%
\subsection{{\bf Generic IDEA }}
We now present an inexact descent algorithm (IDEA) for solving the minimization problem \eqref{eq:dc-problem} via the inexact oracle for HOME-DC function given in Section~\ref{sec:inexactOracle}.
Our algorithm generates an iterative scheme given by
\[
    s^{k+1}=s^k+\alpha_k d^k,
\]
where $\alpha_k>0$ is a step-size, and $d^k$ is a search direction that satisfies the following sufficient descent condition.

\begin{assumption}[Inexact sufficient descent direction]
\label{ass:gradient-related-direction}
There exist constants $\sigma>0$ and $\kappa>0$ such that, for every $k\in \Nz$,
\[
 \langle \widetilde G_k,d^k\rangle  \le -\sigma\|\widetilde G_k\|^2,
\]
and
\[
    \|d^k\|\le \kappa\|\widetilde G_k\|.
\]
\end{assumption}
We summarize the abstract scheme below.  The stopping criteria may be based on the inexact envelope residual $\|\widetilde G_k\|$, on lower-level residuals, or on a prescribed iteration budget.
\vspace{4mm}
%%%%%%%%%%%%%%%%%%%%%%
\begin{algorithm}[h]
\textbf{Initialization:} Choose $s^0\in\R^ n$,~ $\alpha_0\in (0,\overline \alpha)$, and set $k=0$;
\begin{algorithmic}[1]
    \Repeat
        \State Compute approximate proximal points $\widetilde u^k\approx\prox{h}{\gamma}{p}(s^k)$ and $\widetilde v^k\approx\prox{g}{\gamma}{p}(s^k)$;
        \State Form the inexact oracle $\widetilde G_k$ using \eqref{eq:approx-gradient};
        \State Choose a direction $d^k\in \R^n$ satisfying Assumption~\ref{ass:gradient-related-direction};
        \State Choose a step-size $\alpha_k>0$ using a prescribed (dynamic) rule or (inexact) line search;
        \State Set $s^{k+1}=s^k+\alpha_k d^k$ and $k=k+1$;
    \Until{the stopping criterion holds.}
\caption{IDEA (Inexact DEscent Algorithm) \label{alg:idea}}
\end{algorithmic}
\end{algorithm}
%%%%%%%%%%%%%%%
\vspace{4mm}

\begin{remark}[Examples of admissible directions]\label{rem:examples-gradient-related-directions}
The steepest descent direction $d^k=-\widetilde G_k$ satisfies Assumption~\ref{ass:gradient-related-direction} with
$\sigma=\kappa=1$. More generally, if
\[
    d^k=-H_k\widetilde G_k,
\]
where $H_k$ is a symmetric positive  definite matrix and there exist constants
$0<m\le M<+\infty$ such that
\[
    mI\preceq H_k\preceq MI,
\]
then
\[
    \langle \widetilde G_k,d^k\rangle = -\langle \widetilde G_k,H_k\widetilde G_k\rangle \le -m\|\widetilde G_k\|^2
\]
and
\[
    \|d^k\|\le M\|\widetilde G_k\|.
\]
Thus, safeguarded quasi-Newton directions, including L-BFGS-type directions, fall within this framework whenever the generated inverse-Hessian approximations are
uniformly positive definite and uniformly bounded.
\end{remark}

We now state the assumptions for the convergence analysis.

\begin{assumption}[Standing assumptions for the descent analysis]
\label{ass:descent-analysis}
The following conditions hold:
\begin{enumerate}[label=(\textbf{A\arabic*}), leftmargin=0.75cm]
\item \label{ass:descent-analysis:pge2}
$p\ge2$.

\item \label{ass:descent-analysis:bounded}
There exists a bounded convex set $\mathcal B\subset\R^n$ such that, for every $k$,
\[
    [s^k,s^{k+1}]:=\{s^k+t(s^{k+1}-s^k):t\in[0,1]\} \subset \mathcal B.
\]

\item \label{ass:descent-analysis:lipschitz}
The gradient $\nabla\fgam{\gf}{p}{\gamma}$ is Lipschitz continuous on $\mathcal B$ with constant $L>0$.

\item \label{ass:descent-analysis:lower}
$\fgam{\gf}{p}{\gamma}$ is bounded from below on $\R^n$; that is,
$\inft{s\in\R^n}\fgam{\gf}{p}{\gamma}(s)>-\infty$.

\item \label{ass:descent-analysis:error}
The oracle errors satisfy
\[
    \sum_{k=0}^{+\infty}\|e^k\|^2<+\infty.
\]
\end{enumerate}
\end{assumption}
Assumption~\ref{ass:descent-analysis}~\ref{ass:descent-analysis:pge2} is not part of the definition of HOME-DC, which is defined for every $p>1$. It is only a restriction needed for the descent proof. For $p=2$, the classical quadratic smoothness regime is recovered. For $p>2$, Lipschitz continuity of the envelope gradient is available on bounded sets, consistent with Assumption~\ref{ass:descent-analysis}~\ref{ass:descent-analysis:bounded}. For $1<p<2$, the envelope gradient is generally H\"older rather than Lipschitz continuous, so the standard Lipschitz-gradient descent argument used below does not apply.

In the next result, we show that the function values of the HOME-DC function satisfy a nonmonotone descent condition, which is necessary for establishing the convergence of the IDEA framework.

%%%%%%%%%

\begin{lemma}[Perturbed reference-value descent inequality]
\label{lem:perturbed-descent-inequality}
Let Assumptions~\ref{ass:gradient-related-direction} and \ref{ass:descent-analysis} hold. Let $\{M_k\}\subset\mathbb R$ be a reference sequence satisfying
$\fgam{\gf}{p}{\gamma}(s^k) \le M_k$ for all $k\in \Nz$. Assume that the step-sizes satisfy
\[
    0<\underline\alpha\le\alpha_k\le\overline\alpha < \frac{\sigma}{L\kappa^2}.
\]
Then there exist constants $c_0>0$ and $c_1>0$ such that, for all $k\in \Nz$,
\[
\fgam{\gf}{p}{\gamma}(s^{k+1}) \le M_k - c_0\|\widetilde G_k\|^2 + c_1\|e^k\|^2.
\]
More precisely, one can take
$c_0 :=\frac{\underline\alpha}{2} \left(\sigma-L\overline\alpha\kappa^2\right)>0$ and
$c_1:= \frac{\overline\alpha\kappa^2}{2\sigma}$.

\end{lemma}
\begin{proof}
Since $[s^k,s^{k+1}]\subset\mathcal B$ and $\nabla\fgam{\gf}{p}{\gamma}$ is $L$-Lipschitz continuous on $\mathcal B$, the descent lemma gives
\[
\fgam{\gf}{p}{\gamma}(s^{k+1}) \le \fgam{\gf}{p}{\gamma}(s^k) +\alpha_k\langle \nabla\fgam{\gf}{p}{\gamma}(s^k),d^k\rangle
    + \frac{L}{2}\alpha_k^2\|d^k\|^2.
\]
Using $\nabla\fgam{\gf}{p}{\gamma}(s^k)=\widetilde G_k-e^k$, we obtain
\[
\fgam{\gf}{p}{\gamma}(s^{k+1})\le \fgam{\gf}{p}{\gamma}(s^k) + \alpha_k\langle \widetilde G_k,d^k\rangle
    - \alpha_k\langle e^k,d^k\rangle + \frac{L}{2}\alpha_k^2\|d^k\|^2.
\]
Using Assumption~\ref{ass:gradient-related-direction},
\[
\fgam{\gf}{p}{\gamma}(s^{k+1})\le  \fgam{\gf}{p}{\gamma}(s^k)  - \alpha_k\sigma\|\widetilde G_k\|^2  + \alpha_k\kappa\|e^k\|\|\widetilde G_k\|
    +  \frac{L}{2}\alpha_k^2\kappa^2\|\widetilde G_k\|^2.
\]
By Young's inequality,
\[
\alpha_k\kappa\|e^k\|\|\widetilde G_k\| \le \frac{\alpha_k\sigma}{2}\|\widetilde G_k\|^2 + \frac{\alpha_k\kappa^2}{2\sigma}\|e^k\|^2,
\]
leading to
\[
 \fgam{\gf}{p}{\gamma}(s^{k+1})\le \fgam{\gf}{p}{\gamma}(s^k) - \frac{\alpha_k}{2} \left(\sigma-L\alpha_k\kappa^2\right)
    \|\widetilde G_k\|^2 + \frac{\alpha_k\kappa^2}{2\sigma}\|e^k\|^2.
\]
Considering $\alpha_k\ge\underline\alpha$ and $\alpha_k\le\overline\alpha  < \tfrac{\sigma}{L\kappa^2}$,
we obtain
\[
\frac{\alpha_k}{2}\left(\sigma-L\alpha_k\kappa^2\right)\ge\frac{\underline\alpha}{2}\left(\sigma-L\overline\alpha\kappa^2\right)=:c_0>0,
\]
and
\[
\frac{\alpha_k\kappa^2}{2\sigma}\le\frac{\overline\alpha\kappa^2}{2\sigma}=:c_1.
\]
Finally, since $\fgam{\gf}{p}{\gamma}(s^k)\le M_k$, the claimed inequality follows.
\end{proof}

\begin{remark}[Nonmonotone reference values and safeguarded implementations]
\label{rem:gll-nonmonotone-reference}
The reference sequence $M_k$ in Lemma~\ref{lem:perturbed-descent-inequality} is arbitrary except for the requirement $\fgam{\gf}{p}{\gamma}(s^k)\le M_k$. The monotone choice $M_k=\fgam{\gf}{p}{\gamma}(s^k)$ recovers the usual perturbed descent inequality.
In this monotone case, Assumption~\ref{ass:descent-analysis}~\ref{ass:descent-analysis:lower} and Assumption~\ref{ass:descent-analysis}~\ref{ass:descent-analysis:error} imply that $\{\fgam{\gf}{p}{\gamma}(s^k)\}$ has a finite limit, because
\[
\fgam{\gf}{p}{\gamma}(s^{k+1}) \le \fgam{\gf}{p}{\gamma}(s^k)+c_1\|e^k\|^2
\]
and $\sum_k\|e^k\|^2<+\infty$.

A finite-memory Grippo--Lampariello--Lucidi \cite{grippo1986nonmonotone} reference value is obtained by setting
\begin{equation}\label{eq:M_k}
M_k:=\maxt{0\le j\le \min\{m,k\}}\fgam{\gf}{p}{\gamma}(s^{k-j}),\qquad m\in\mathbb N_0.
\end{equation}
For more nonmonotone strategies, we refer to \cite{ahookhosh2017efficiency,ahookhosh2012class,amini2014inexact,zhang2004nonmonotone}. 
Other nonmonotone reference values may also be used, provided that they dominate $\fgam{\gf}{p}{\gamma}(s^k)$. This flexibility is useful for safeguarded quasi-Newton and L-BFGS implementations, where strict monotonicity may be unnecessarily restrictive.
\end{remark}

Next, we establish the convergence of our generic descent method.
\begin{theorem}[Asymptotic stationarity of inexact gradient-related descent]
\label{thm:reference-value-stationarity}
Let Assumptions~\ref{ass:gradient-related-direction} and \ref{ass:descent-analysis} hold. Let $\{M_k\}\subset\mathbb R$ be a reference sequence satisfying
$\fgam{\gf}{p}{\gamma}(s^k) \le M_k$ for all $k\in \Nz$. Assume that the step-sizes satisfy
\[
    0<\underline\alpha\le\alpha_k\le\overline\alpha < \frac{\sigma}{L\kappa^2}.
\]
Assume further that both limits exist and  $\lim_{k\to\infty}\fgam{\gf}{p}{\gamma}(s^k)=\lim_{k\to\infty}M_k$.
Then $\widetilde G_k\to0$ and $\nabla\fgam{\gf}{p}{\gamma}(s^k)\to 0$.
\end{theorem}
\begin{proof}
Rearranging the inequality in Lemma~\ref{lem:perturbed-descent-inequality} gives
\[
c_0\|\widetilde G_k\|^2\le M_k-\fgam{\gf}{p}{\gamma}(s^{k+1})+c_1\|e^k\|^2.
\]
Since $e^k\to 0$, we obtain $\widetilde G_k\to0$. Since $\nabla\fgam{\gf}{p}{\gamma}(s^k)=\widetilde G_k-e^k$,
we get $\nabla\fgam{\gf}{p}{\gamma}(s^k)\to0$.
\end{proof}

%%%%%%%%%
\begin{corollary}[Stationarity of cluster points]
\label{cor:cluster-points-stationary}
Let the assumptions of Theorem~\ref{thm:reference-value-stationarity}
hold. If $\ov s$ is a cluster point of $\{s^k\}$, then $\nabla\fgam{\gf}{p}{\gamma}(\ov s)=0$.
Consequently, setting $\ov x:=u(\ov s)=v(\ov s)$, where $u=\prox{h}{\gamma}{p}$ and $v=\prox{g}{\gamma}{p}$, gives $\ov x\in\stat(\gf)$.
\end{corollary}
\begin{proof}
Let $s^{k_j}\to\ov s$. By Theorem~\ref{thm:reference-value-stationarity}, $\nabla\fgam{\gf}{p}{\gamma}(s^{k_j})\to 0.$ Since $\nabla\fgam{\gf}{p}{\gamma}$ is continuous, we obtain $\nabla\fgam{\gf}{p}{\gamma}(\ov s)=0$. Theorem~\ref{thm:critical-lift-description} then gives 
$u(\ov s)=v(\ov s)$ and $\ov x:=u(\ov s)=v(\ov s)\in\stat(\gf)$.
\end{proof}

%%%%%%%%
\begin{remark}[Relation with gradient descent and L-BFGS]
\label{rem:relation-gradient-lbfgs}
The preceding analysis is not tied to the steepest-descent direction. If $d^k=-\widetilde G_k$, then the method reduces to an inexact gradient method for
the difference of high-order Moreau envelopes. If $d^k=-H_k\widetilde G_k$ with uniformly bounded positive definite matrices $H_k$, then the method covers damped quasi-Newton-type
directions. In numerical experiments, one may use an L-BFGS direction as an acceleration strategy, provided that the accepted direction is safeguarded to
satisfy Assumption~\ref{ass:gradient-related-direction}. If the L-BFGS direction fails the descent test, the steepest-descent direction $-\widetilde G_k$ can be used
as a fallback.
\end{remark}

The preceding theorem assumes a stepsize sequence satisfying a uniform upper bound. In practice, such a bound is usually not known. We therefore record a line-search counterpart based on a finite-memory nonmonotone Armijo rule. The result shows that, under the same gradient-related direction condition and a relative oracle-error bound, the backtracking procedure is well-defined, the accepted stepsizes remain uniformly positive, and the generated sequence is asymptotically stationary.

\begin{theorem}[Asymptotic stationarity under nonmonotone Armijo line search]\label{thm:armijo}
Let Assumptions~\ref{ass:gradient-related-direction} and \ref{ass:descent-analysis} hold.
Let $\rho\in(0,1)$,
$\beta\in(0,1)$, and $\ov\alpha>0$. For a fixed memory length $m\in\Nz$, set
$M_k$ by \eqref{eq:M_k}. Given $s^k$, $\widetilde G_k$, and $d^k$, choose $\alpha_k$ as the first accepted element, in backtracking order, of the sequence $\{\beta^j\ov\alpha\}_{j\in \Nz}$ such that
\begin{equation}\label{eq:armijo-inexact-condition}
\fgam{\gf}{p}{\gamma}(s^k+\alpha_k d^k) \le M_k+\rho\alpha_k\inner{\widetilde G_k}{d^k}+\delta_k,
\end{equation}
where $\delta_k\ge0$ and $\sum_{k=0}^{+\infty}\delta_k<+\infty$.
Assume, in addition, that the oracle error satisfies the relative bound
$\|e^k\|\le \eta\|\widetilde G_k\|$ for all sufficiently large $k$,
where $0\le\eta<\frac{(1-\rho)\sigma}{\kappa}$.
Then the following statements hold.
\begin{enumerate}[label=(\textbf{\alph*}), font=\normalfont\bfseries, leftmargin=0.7cm]
\item \label{thm:armijo:a} The line search is well-defined for all sufficiently large $k$. 

\item \label{thm:armijo:b} There exists $\alpha_{\min}>0$ such that
$\alpha_k\ge \alpha_{\min}$ for all sufficiently large $k$.

\item \label{thm:armijo:c} If $\fgam{\gf}{p}{\gamma}$ is bounded from below, then
$\sum_{k=0}^{+\infty}\|\widetilde G_k\|^2<+\infty$, $\widetilde G_k\to 0$, $e^k\to 0$, and
$\nabla \fgam{\gf}{p}{\gamma}(s^k)\to 0$.

\item \label{thm:armijo:d} Every cluster point $\ov s$ of $\{s^k\}_{k\in \Nz}$ satisfies
$\nabla \fgam{\gf}{p}{\gamma}(\ov s)=0$. Hence, setting
$\ov x:=\prox{g}{\gamma}{p}(\ov s)=\prox{h}{\gamma}{p}(\ov s)$,
one has $\ov x\in\stat(\gf)$.
\end{enumerate}
\end{theorem}
\begin{proof}
\ref{thm:armijo:a}
Let $k$ be sufficiently large so that $\|e^k\|\le \eta\|\widetilde G_k\|$. Since $\nabla \fgam{\gf}{p}{\gamma}$ is Lipschitz continuous on the bounded set containing
the relevant line segments, the descent lemma gives, for every $\alpha>0$ such
that $[s^k,s^k+\alpha d^k]\subset B$,
\[
\fgam{\gf}{p}{\gamma}(s^k+\alpha d^k)\le\fgam{\gf}{p}{\gamma}(s^k)+\alpha\inner{\nabla \fgam{\gf}{p}{\gamma}(s^k)}{d^k}+\frac{L}{2}\alpha^2\|d^k\|^2.
\]
Using $\nabla \fgam{\gf}{p}{\gamma}(s^k)=\widetilde G_k-e^k$, we obtain
\[
\begin{aligned}
\fgam{\gf}{p}{\gamma}(s^k+\alpha d^k)&\le\fgam{\gf}{p}{\gamma}(s^k)+\alpha\inner{\widetilde G_k}{d^k}-\alpha\inner{e^k}{d^k}+\frac{L}{2}\alpha^2\|d^k\|^2\\
&=\fgam{\gf}{p}{\gamma}(s^k)+\rho \alpha\inner{\widetilde G_k}{d^k}+(1-\rho)\alpha\inner{\widetilde G_k}{d^k}-\alpha\inner{e^k}{d^k}+\frac{L}{2}\alpha^2\|d^k\|^2.
\end{aligned}
\]
By Assumption~\ref{ass:gradient-related-direction},
\[
(1-\rho)\alpha\inner{\widetilde G_k}{d^k}\le-(1-\rho)\alpha\sigma\norm{\widetilde G_k}^2,
\]
and
\[
\frac{L}{2}\alpha^2\norm{d^k}^2\le\frac{L}{2}\alpha^2\kappa^2\norm{\widetilde G_k}^2.
\]
Moreover,
\[
-\alpha\inner{e^k}{d^k}\le\alpha\norm{e^k}\norm{d^k}\le\alpha\eta\kappa\norm{\widetilde G_k}^2,
\]
i.e.,
\[
\fgam{\gf}{p}{\gamma}(s^k+\alpha d^k)\le\fgam{\gf}{p}{\gamma}(s^k)+\rho\alpha\inner{\widetilde G_k}{d^k}+\alpha\left[-(1-\rho)\sigma+\eta\kappa+\frac{L}{2}\alpha\kappa^2\right]\|\widetilde G_k\|^2.
\]
Hence, the Armijo condition \eqref{eq:armijo-inexact-condition} is satisfied
whenever
\[
0<\alpha\le\frac{2\big((1-\rho)\sigma-\eta\kappa\big)}{L\kappa^2},
\]
with the usual convention that the upper bound is $+\infty$ if $L=0$. Since
$M_k\ge \fgam{\gf}{p}{\gamma}(s^k)$ and $\delta_k\ge0$, and the backtracking sequence tends to zero, the backtracking procedure terminates.
\\
\ref{thm:armijo:b} By the construction of the backtracking rule,
\[
\alpha_k\ge\alpha_{\min}:=\beta\min\left\{\ov\alpha,\frac{2\big((1-\rho)\sigma-\eta\kappa\big)}{L\kappa^2}\right\}>0
\]
for all sufficiently large $k$.
\\
\ref{thm:armijo:c}
Using \eqref{eq:armijo-inexact-condition} and
$\inner{\widetilde G_k}{d^k}\le -\sigma\|\widetilde G_k\|^2$, we get
\[
\fgam{\gf}{p}{\gamma}(s^{k+1})\le M_k-\rho\sigma\alpha_k\|\widetilde G_k\|^2+\delta_k
\le M_k-c|\widetilde G_k|^2+\delta_k,
\]
where $c:=\rho\sigma\alpha_{\min}>0$. The summability of perturbations ${\delta_k}$
implies
\[
\sum_{k=0}^{+\infty}\|\widetilde G_k\|^2<+\infty.
\]
Therefore $\widetilde G_k\to 0$. The relative error condition gives $e^k\to 0$, and hence $\nabla \fgam{\gf}{p}{\gamma}(s^k)=\widetilde G_k-e^k\to 0$.
\\
\ref{thm:armijo:d}
If $\ov s$ is a cluster point, continuity of $\nabla \fgam{\gf}{p}{\gamma}$ gives 
$\nabla \fgam{\gf}{p}{\gamma}(\ov s)=0$. The final statement follows from the criticality
correspondence between HOME-DC critical centers and DC-stationary points.
\end{proof}
%%%%%%%%%%%%%%%%%%%%%%%%%%%%%%%%%%%%%%%%%%%%%%%%%%%%%%%%%%%%%%%%%%%%%%%%%%%%%%%%%%%%%%%%%%%%%%%%%

%%%%%%%%%%%%%%%%%%%%%%%%%%%%%%%%%%%%%%%%%%%%%%%%%%%%%%%%%%%%%%%%%%%%%%%%%%%%%%%%%%%%%%%%%%%%%%%%%%%%%%%%%%%%%%%%%%%%%%%%%%%%%%%%%%%%%%%%%%%%%%%%%%%%
%%% End Section: Inexact high-order two-prox methods %%%%%%%%%%%%%%%%%%%%%%%%%%%%%%%%%%%%%%%%%%%%%%%%%%%%%%%%%%%%%%%%%%%%%%%%%%%%%%%%%%%%%%%%%%%%%%%%%

%%%%%%%%%%%%%%%%%%%%%%%%%%%%%%%%%%%%%%%%%%%%%%%%%%%%%%%%%%%%%%%%%%%%%%%%%%%%%%%%%%%%%%%%%%%%%%%%%%%%%%%%%%%%%%%%%%%%%%%%%%%%%%%%%%%%%%%%%%%%%%%%%%%%
%%% Section: Residual-based inexactness %%%%%%%%%%%%%%%%%%%%%%%%%%%%%%%%%%%%%%%%%%%%%%%%%%%%%%%%%%%%%%%%%%%%%%%%%%%%%%%%%%%%%%%%%%%%%%%%%
\section{Residual-Based Inexactness for Lower-Level Subproblems}\label{sec:residual}
In this section, we show how the abstract envelope-gradient error conditions in Section~\ref{sec:algorithm} can be enforced through computable residuals of the two lower-level high-order proximal subproblems.

Let $\psi:\R^n\to\Rinf$ be proper, lsc, and convex. For a given center $s\in\R^n$, the exact high-order proximal point
\[
    x=\prox{\psi}{\gamma}{p}(s)
\]
is characterized by
\[
    0\in \partial\psi(x)+\frac{1}{\gamma} J_p(x-s).
\]
Thus, if $\widetilde x$ is an approximate solution, a natural residual is any vector
$\eta$ that satisfies
\begin{equation}\label{eq:lower-residual-convention}
\eta\in \partial\psi(\widetilde x)+\frac{1}{\gamma} J_p(\widetilde x-s).
\end{equation}
Thus $\eta=0$ is precisely the exact optimality condition for the lower-level proximal subproblem.

The following estimate converts this residual bound into a bound on the distance between the approximate and exact proximal points.

\begin{lemma}[Residual controls the proximal-point error]\label{lem:residual-distance}
Let $p\ge 2$, $\gamma>0$, and let $\gh:\R^n\to\Rinf$ be proper, lsc, and convex.
Suppose that $x=\prox{\gh}{\gamma}{p}(s)$ and that $\widetilde x\in\R^n$ and $\eta\in\R^n$ satisfy
\eqref{eq:lower-residual-convention}.
Then there exists a constant $C_{p,\gamma}>0$, depending only on $p$ and $\gamma$, such that
\[
    \|\widetilde x-x\| \le C_{p,\gamma}\|\eta\|^{1/(p-1)}.
\]
\end{lemma}
\begin{proof}
By Lemma~\ref{lem:optimality}, it holds that
\[
    -\frac{1}{\gamma} J_p(x-s)\in \partial \gh(x).
\]
Choose $\xi\in\partial \gh(x)$ and $\widetilde\xi\in\partial \gh(\widetilde x)$ such that
\[
    \xi+\frac{1}{\gamma} J_p(x-s)=0,
    \quad
    \widetilde\xi+\frac{1}{\gamma} J_p(\widetilde x-s)=\eta.
\]
Subtracting the first identity from the second yields
\[
    (\widetilde\xi-\xi)+\frac{1}{\gamma}\bigl(J_p(\widetilde x-s)-J_p(x-s)\bigr)=\eta.
\]
Taking the inner product with $\widetilde x-x$ gives
\[
    \inner{\widetilde\xi-\xi}{\widetilde x-x}+\frac{1}{\gamma}
    \inner{J_p(\widetilde x-s)-J_p(x-s)}{\widetilde x-x}
    =\inner{\eta}{\widetilde x-x}.
\]
Since $\partial \gh$ is monotone, $\langle\widetilde\xi-\xi , \widetilde x-x\rangle\ge 0$, i.e.,
\begin{equation}\label{eq:mono-step}
    \frac{1}{\gamma}\inner{J_p(\widetilde x-s)-J_p(x-s)}{\widetilde x-x}
    \le\norm{\eta}\,\norm{\widetilde x-x}.
\end{equation}
For $p\ge 2$, the map $J_p$ is $p$-uniformly monotone (see \cite[Lemma~4.2.3]{Nesterov2018}), i.e, there exists $c_p>0$ such that
\[
    \inner{J_p(a)-J_p(b)}{a-b}\ge c_p\norm{a-b}^p, \qquad \forall a,b\in\R^n.
\]
Applying this with $a=\widetilde x-s$ and $b=x-s$, and combining with \eqref{eq:mono-step}, we obtain
\[
    \frac{c_p}{\gamma}\norm{\widetilde x-x}^p\le \norm{\eta}\,\norm{\widetilde x-x}.
\]
If $\widetilde x=x$, the claim is immediate. Otherwise, dividing by $\|\widetilde x-x\|$ gives
\[
    \norm{\widetilde x-x}^{p-1}\le \frac{\gamma}{c_p}\norm{\eta},
\]
which shows our desired result.
\end{proof}

\begin{lemma}[Residuals control the envelope-gradient error]\label{lem:residual-gradient}
Let $g,h$ satisfy Assumption~\ref{ass:basic}~\ref{ass:basic:a}. Let $p\ge2$, $\gamma>0$, and let $B\subseteq\R^n$ be bounded.
Then there exists a constant $C_{B,p,\gamma}>0$ such that 
if $u=\prox{h}{\gamma}{p}(s)$ and $v=\prox{g}{\gamma}{p}(s)$, with $s,u,v,\widetilde u,\widetilde v\in B$, satisfy
\[
    \eta_h\in \partial h(\widetilde u)+\frac{1}{\gamma} J_p(\widetilde u-s),
    \qquad \eta_g\in \partial g(\widetilde v)+\frac{1}{\gamma} J_p(\widetilde v-s),
\]
then
\begin{equation}\label{eq:residual-gradient-transfer}
    \left\|\frac{1}{\gamma}\Bigl(J_p(s-\widetilde v)-J_p(s-\widetilde u)\Bigr)
        -\frac{1}{\gamma}\Bigl(J_p(s-v)-J_p(s-u)\Bigr)\right\|
    \le C_{B,p,\gamma}\Bigl(\norm{\eta_g}^{1/(p-1)}+\norm{\eta_h}^{1/(p-1)}\Bigr).
\end{equation}
\end{lemma}
\begin{proof}
By Lemma~\ref{lem:residual-distance}, there exists a constant $C_{p,\gamma}>0$ such that
\[
    \norm{\widetilde v-v}\le C_{p,\gamma}\norm{\eta_g}^{1/(p-1)},
    \qquad
    \norm{\widetilde u-u}\le C_{p,\gamma}\norm{\eta_h}^{1/(p-1)}.
\]
Since $B$ is bounded and $p\ge 2$, the map $J_p$ is Lipschitz continuous on $B-B$.
Hence, there exists $L_{B,p}>0$ such that
\[
    \norm{J_p(s-\widetilde v)-J_p(s-v)}\le L_{B,p}\norm{\widetilde v-v},
\]
and
\[
    \norm{J_p(s-\widetilde u)-J_p(s-u)}\le L_{B,p}\norm{\widetilde u-u}.
\]
Therefore
\begin{align*}
    &\left\|\frac{1}{\gamma}\Bigl(J_p(s-\widetilde v)-J_p(s-\widetilde u)\Bigr)
        -\frac{1}{\gamma}\Bigl(J_p(s-v)-J_p(s-u)\Bigr)\right\| \\
    &\quad\quad\quad\quad\le\frac{1}{\gamma}\norm{J_p(s-\widetilde v)-J_p(s-v)}
    +\frac{1}{\gamma}\norm{J_p(s-\widetilde u)-J_p(s-u)} \\
    &\quad\quad\quad\quad\le \frac{L_{B,p} C_{p,\gamma}}{\gamma}
    \Bigl(\norm{\eta_g}^{1/(p-1)}+\norm{\eta_h}^{1/(p-1)}\Bigr),
\end{align*}
which proves \eqref{eq:residual-gradient-transfer}.
\end{proof}

\begin{corollary}[Practical summable inexactness criterion]\label{cor:practical-inexactness}
Assume the hypotheses of Lemma~\ref{lem:residual-gradient}, and suppose that for each $k\in\Nz$, the
approximate proximal points $\widetilde u^k,\widetilde v^k$ satisfy
\[
    \eta_h^k\in \partial h(\widetilde u^k)+\frac{1}{\gamma} J_p(\widetilde u^k-s^k),
    \qquad
    \eta_g^k\in \partial g(\widetilde v^k)+\frac{1}{\gamma} J_p(\widetilde v^k-s^k).
\]
If
\begin{equation}\label{eq:practical-inexactness}
    \sum_{k=0}^\infty
    \left(\norm{\eta_g^k}^{2/(p-1)}+\norm{\eta_h^k}^{2/(p-1)}\right)
    <\infty,
\end{equation}
then 
\[
    \sum_{k=0}^{+\infty}\|e^k\|^2<+\infty.
\]
Consequently, provided the remaining assumptions of Theorem~\ref{thm:reference-value-stationarity} and Corollary~\ref{cor:cluster-points-stationary} hold, their conclusions apply.
\end{corollary}
\begin{proof}
By Lemma~\ref{lem:residual-gradient}, there exists a constant $C>0$ such that
\[
    \norm{e^k}\le C\Bigl(\norm{\eta_g^k}^{1/(p-1)}+\norm{\eta_h^k}^{1/(p-1)}\Bigr).
\]
Squaring, using $(a+b)^2\le2(a^2+b^2)$, and summing over $k$ gives the result.
\end{proof}

\begin{corollary}[Residual-based recovery for $p\ge 2$]\label{cor:residual-recovery}
Assume the hypotheses of Theorem~\ref{thm:inexact-primal-recovery}, and suppose $p\ge 2$.
For every $k\in\Nz$, let the approximate proximal points satisfy
\[
    \eta_h^k\in \partial h(\widetilde u^k)+\frac{1}{\gamma} J_p(\widetilde u^k-s^k),
    \quad
    \eta_g^k\in \partial g(\widetilde v^k)+\frac{1}{\gamma} J_p(\widetilde v^k-s^k),
\]
with $\norm{\eta_h^k}\to 0$ and $\norm{\eta_g^k}\to 0$.
Then
\[
    \dist(\widetilde u^k,\mathcal X^\star)\to 0,
    \quad \dist(\widetilde v^k,\mathcal X^\star)\to 0.
\]
\end{corollary}
\begin{proof}
By Lemma~\ref{lem:residual-distance},
\[
    \|\widetilde u^k-u^k\| \le C_{p,\gamma}\|\eta_h^k\|^{1/(p-1)}\to 0,
    \qquad
    \|\widetilde v^k-v^k\|  \le C_{p,\gamma}\|\eta_g^k\|^{1/(p-1)}\to 0.
\]
The conclusion follows from Theorem~\ref{thm:inexact-primal-recovery}.
\end{proof}

\begin{remark}[Interpretation of the residual rule]
\label{rem:interpretation-residual-rule}
Corollary~\ref{cor:practical-inexactness} provides the basis for an implementable stopping rule for lower-level proximal subproblems. It shows that the proximal
 residuals need not vanish faster than required by the summability condition. Instead, it is enough to impose the summability condition
\[
\sum_{k=0}^{+\infty} \left(\|\eta_g^k\|^{2/(p-1)}+\|\eta_h^k\|^{2/(p-1)}\right)<+\infty.
\]
For $p=2$, this reduces to the square summability of the lower-level residuals. For $p>2$, the exponent $2/(p-1)$ reflects the nonlinear high-order proximal regularization and is one of the main differences from the quadratic difference-of-Moreau-envelopes setting.
\end{remark}

\begin{remark}[Connection with the numerical implementation]\label{rem:connection-numerical-lbfgs}
The numerical section may use L-BFGS to generate the search direction $d^k$ from the approximate gradient $\widetilde G_k$. The present analysis does not require
the direction to be exactly the negative gradient. It requires only the gradient-related conditions in Assumption~\ref{ass:gradient-related-direction} and
the summable gradient-error condition verified by Corollary~\ref{cor:practical-inexactness}. Thus, L-BFGS should be
understood as a practical acceleration strategy within the inexact gradient-related framework. A safeguarded implementation can enforce the required descent condition by rejecting an L-BFGS direction that fails the inexact descent condition and using $-\widetilde G_k$ instead.
\end{remark}
%%%%%%%%%%%%%%%%%%%%%%%%%%%%%%%%%%%%%%%%%%%%%%%%%%%%%%%%%%%%%%%%%%%%%%%%%%%%%%%%%%%%%%%%%%%%%%%%%%%%%%%%%%%%%%%%%%%%%%%%%%%%%%%%%%%%%%%%%%%%%%%%%%%%
%%% End Section: Residual-based inexactness %%%%%%%%%%%%%%%%%%%%%%%%%%%%%%%%%%%%%%%%%%%%%%%%%%%%%%%%%%%%%%%%%%%%%%%%%%%%%%%%%%%%%%%%%%%%%%%%%

%%%%%%%%%%%%%%%%%%%%%%%%%%%%%%%%%%%%%%%%%%%%%%%%%%%%%%%%%%%%%%%%%%%%%%%%%%%%%%%%%%%%%%%%%%%%%%%%%
%%%%%%%%%%%%%%%%%%%%%%%%%%%%%%%%%%%%%%%%%%%%%%%%%%%%%%%%%%%%%%%%%%%%%%%%%%%%%%%%%%%%%%%%%%%%%%%%%
\section{Preliminary Numerical Experiments}\label{sec:numerical}
In this section, we investigate the numerical performance of IDEA and compare it with the classical DCA. Since the primary focus of this paper is theoretical, we do not aim to provide an extensive numerical study; rather, we present a limited set of experiments to illustrate the practical behavior of the proposed method.

As a benchmark problem, we consider the clustering problem of partitioning a dataset
\[
A=\{a_1,\ldots,a_m\}\subset\mathbb{R}^n
\]
into \(k\) clusters. Each cluster is represented by a center \(x_j \in \mathbb{R}^n\), \(j=1,\ldots,k\), and each data point is assigned to the nearest center. A standard formulation introduces binary assignment variables
$u_{ij}\in\{0,1\}$, 
where \(u_{ij}=1\) if the data point \(a_i\) is assigned to the cluster center \(x_j\), and \(u_{ij}=0\) otherwise. The corresponding mixed-integer optimization problem is
\[
\begin{aligned}
\min_{x,u}\quad &
\frac{1}{m}\sum_{i=1}^{m}\sum_{j=1}^{k}
u_{ij}\|x_j-a_i\| \\
\text{s.t.}\quad &
\sum_{j=1}^{k}u_{ij}=1,
\qquad i=1,\ldots,m,\\
&
u_{ij}\in\{0,1\},
\qquad i=1,\ldots,m,\;\; j=1,\ldots,k,
\end{aligned}
\]
In conventional clustering, the squared Euclidean norm is typically used to measure the distance between data points and cluster centers due to its smoothness and strong convexity properties, facilitating efficient optimization. In contrast, we adopt the Euclidean norm, leading to a more robust formulation that is less sensitive to outliers and thus better preserves the structure of data in the presence of noise or aberrant observations.

For any fixed collection of centers
$X=(x_1,\ldots,x_k)$, 
the optimal assignment is obtained by assigning each point \(a_i\) to its nearest center. Consequently, as shown in \cite{bock1998clustering, ordin2015heuristic}, the binary assignment variables can be eliminated explicitly, which yields the equivalent unconstrained formulation,
\begin{align}\label{CP}
\min\
\tfrac{1}{m}\sum_{i=1}^{m}
\min_{j=1,\ldots,k}\|x_j-a_i\|,
\tag{CP}
\end{align}
which has the same optimal value and optimal cluster centers as the mixed-integer formulation.
 Although the two formulations are equivalent in the sense that they share the same optimal value and optimal cluster centers, the latter avoids the combinatorial binary variables and reveals the intrinsic nonsmooth and nonconvex structure of the clustering problem. Furthermore, it admits a natural DC decomposition, making it particularly suitable for evaluating the performance of the proposed algorithm and comparing it with the classical DCA.
\[
f_i(X):=\mint{j=1,\ldots,k} \|x_j-a_i\|.
\]
Using the identity
\[
\mint{j=1,\ldots,k} t_j=\sum_{j=1}^k t_j-\maxt{r=1,\ldots,k}
\sum_{\substack{j=1 \\ j\neq r}}^k t_j,
\]
we obtain
\[
f_i(X)=\sum_{j=1}^k \|x_j-a_i\|-\maxt{r=1,\ldots,k}\sum_{\substack{j=1 \\ j\neq r}}^k \|x_j-a_i\|.
\]
Therefore, problem \eqref{CP} can be written as the DC program
\begin{equation}
\mint{X\in\mathbb{R}^{n\times k}}\; g(X)-h(X), \quad
g(X)=\frac{1}{m}\sum_{i=1}^m\sum_{j=1}^k \|x_j-a_i\|, \quad
h(X)=\frac{1}{m}\sum_{i=1}^m\maxt{r=1,\ldots,k}\sum_{\substack{j=1 \\ j\neq r}}^k \|x_j-a_i\|.
\label{eq:dc_clustering}
\end{equation}
where both $g$ and $h$ are proper, closed, and convex; hence, \eqref{eq:dc_clustering} is a valid DC decomposition of the clustering problem.
%%%%%%%%%%%%%%%%%%%%%%%%%%%%%%%%%%%%%%%%%%%%%%%%%%%%%%%%%%%%%%%%%%%%%%%

To evaluate the performance and robustness of the proposed envelope-based algorithm, we conducted numerical experiments on six real-world datasets sourced from the \textit{scikit-learn} Python library \cite{pedregosa2011scikit}. These datasets exhibit diverse characteristics in terms of dimensionality ($n$), number of samples ($m$), and spatial structure, thereby providing a comprehensive benchmark for evaluating clustering performance. A summary of the datasets is presented in Table~\ref{tab:datasets}.

\vspace{-5mm}
\begin{table}[htbp]
\centering
\caption{Summary of datasets}
\label{tab:datasets}
\begin{tabular}{lcc}
\toprule
\textbf{Dataset} & \textbf{Samples ($m$)} & \textbf{Features ($n$)} \\
\midrule
Breast Cancer & 569   & 30 \\
Diabetes      & 442   & 10 \\
Digits        & 1,797 & 64 \\
Iris          & 150   & 4  \\
Linnerud      & 20    & 3  \\
Wine          & 178   & 13 \\
\bottomrule
\end{tabular}
\end{table}

We compared the DCA with the proposed high-order Moreau envelope smoothing approach. The minimization of the envelope function was carried out using the L-BFGS algorithm. To investigate the influence of the smoothing parameter, we considered three values of the exponent parameter, namely $p \in \{1.5, 2.0, 3.0\}$, while fixing the step-size parameter at $\gamma = 1$ in all experiments. 

It is important to note that the boosted version of DCA is not applicable in this setting, as it was originally developed under the assumption that the function $g$ is smooth \cite{aragon2020boosted}, which does not hold for the problem considered here.

To ensure a fair and consistent comparison, all algorithms were initialized from the exact same starting centers, $X_0 \in \mathbb{R}^{n \times k}$, which were generated randomly from a normal distribution scaled by the variance of the respective dataset. The experiments were repeated for different numbers of target clusters, specifically $k \in \{4, 7, 10\}$. Moreover, all numerical experiments were performed on a Windows laptop with a 12th Gen Intel Core i7-12700H CPU (2.30 GHz) and 32 GB RAM.

For the considered problem, neither the DCA subproblems nor the high-order Moreau envelope subproblems admit closed-form solutions. For the standard DCA, the convex subproblem at each iteration was solved utilizing the BFGS optimizer. For the proposed envelope method, computing the objective function and its gradient requires evaluating the proximal operators for $g$ and $h$. These nested optimization problems were also solved using the standard BFGS algorithm. Both the DCA subproblem and the proposed envelope method were set to terminate when the absolute difference in the objective value between two consecutive iterations fell below a tolerance of $\epsilon = 10^{-6}$. All implementation codes are publicly available in the GitHub repository, \url{https://github.com/molsemzamani/Home_DCA}.

\vspace{-5mm}
\subsection{{\bf Overview of Efficiency and Performance}}
Table \ref{tab:numerical_results} summarizes the final optimal objective values, the number of outer iterations, and the total computational runtime (in seconds) for each configuration. 

\textbf{Solution Quality:}  The most notable advantage of the proposed envelope method is its robust capacity to escape poor local minima. The clustering formulation \eqref{CP} is highly nonconvex, causing standard DCA to frequently stall at suboptimal configurations. For instance, on the  \textbf{Digits} dataset with $k=7$, DCA stagnated at an objective value of $30.83$, whereas the proposed method ($p=1.5$) successfully navigated the landscape to reach a significantly better minimum of $29.59$. A similar breakthrough is observed in the \textbf{Wine} dataset for $k=10$, where the proposed method ($p=1.5$) achieved an objective value of $53.98$, drastically outperforming the DCA's terminal value of $131.39$.

\textbf{Computational Efficiency:} 
Table~\ref{tab:numerical_results} and Figure~\ref{fig:performance_time} illustrate the computational efficiency of the algorithms based on total CPU time. While the performance profile indicates standard DCA has a high probability of being the fastest solver initially, the table reveals this is primarily due to its rapid execution on simpler datasets like Diabetes and Iris (often completing in under 0.1s). However, DCA struggles significantly with scalability on more demanding problems, requiring up to 50.49 seconds for the Breast Cancer dataset and spiking to 105.67 seconds for Digits ($k=7$).

In contrast, the proposed envelope methods with $p \ge 2$ offer vastly superior consistency and speed on these challenging datasets. Specifically, L-BFGS ($p=2.0$) drastically reduces computational overhead, solving the same Breast Cancer and Digits instances in just a few seconds (under 8s and 6s, respectively). Conversely, the $p=1.5$ variant remains significantly slower, requiring the highest number of iterations and CPU time across most datasets (e.g., nearly 60s on Digits $k=10$). These results demonstrate that choosing $p=2.0$ or $p=3.0$ not only prevents the drastic runtime spikes seen in DCA but optimally balances overall computational efficiency with improved objective outcomes.

\textbf{Impact of Parameter $p$:} The proximity parameter $p$ governs the trade-off between smoothing and retaining the original geometry. Using $p=2.0$ or $p=3.0$ generally yields the highest computational efficiency and fastest runtimes. Using $p=1.5$ requires the highest computational effort (resulting in longer runtimes and more iterations) but preserves sharper geometric features is some cases, which occasionally allows it to uncover the absolute best local minima (as demonstrated in the Wine and Linnerud datasets). 

\vspace{-4mm}
\begin{table}[htbp]
\centering
\caption{Numerical Results for Clustering Experiments}
\label{tab:numerical_results}
\resizebox{\textwidth}{!}{%
\begin{tabular}{ll|rrr|rrr|rrr|rrr}
\toprule
& & \multicolumn{3}{c|}{\textbf{DCA}} & \multicolumn{3}{c|}{\textbf{L-BFGS ($p=1.5$)}} & \multicolumn{3}{c|}{\textbf{L-BFGS ($p=2.0$)}} & \multicolumn{3}{c}{\textbf{L-BFGS ($p=3.0$)}} \\
\cmidrule(lr){3-5} \cmidrule(lr){6-8} \cmidrule(lr){9-11} \cmidrule(lr){12-14}
\textbf{Dataset} & $\boldsymbol{k}$ & \textbf{Value} & \textbf{Time (s)} & \textbf{Iter} & \textbf{Value} & \textbf{Time (s)} & \textbf{Iter} & \textbf{Value} & \textbf{Time (s)} & \textbf{Iter} & \textbf{Value} & \textbf{Time (s)} & \textbf{Iter} \\
\midrule

\multirow{3}{*}{Breast Cancer} 
& 4  & 172.07 & 10.21 & 36 & 165.19 & 3.49 & 51 & 165.19 & 1.97 & 53 & 165.19 & 3.01 & 47 \\
& 7  & 143.12 & 31.26 & 49 & 129.28 & 13.47 & 82 & 129.28 & 4.70 & 76 & 129.28 & 10.12 & 82 \\
& 10 & 143.12 & 50.49 & 49 & 129.26 & 19.62 & 69 & 130.05 & 7.65 & 76 & 129.28 & 18.19 & 88 \\
\midrule

\multirow{3}{*}{Diabetes} 
& 4  & 0.14 & 0.05 & 3 & 0.18 & 2.96 & 14 & 0.14 & 0.54 & 4 & 0.14 & 0.53 & 4 \\
& 7  & 0.14 & 0.06 & 3 & 0.32 & 8.41 & 20 & 0.14 & 1.05 & 4 & 0.14 & 0.74 & 5 \\
& 10 & 0.14 & 0.05 & 3 & 0.46 & 33.65 & 51 & 0.12 & 5.80 & 14 & 0.14 & 1.26 & 6 \\
\midrule

\multirow{3}{*}{Digits} 
& 4  & 32.46 & 4.30 & 26 & 34.47 & 3.51 & 7 & 34.47 & 2.24 & 6 & 34.47 & 2.02 & 6 \\
& 7  & 30.83 & 105.67 & 165 & 29.59 & 31.30 & 37 & 29.67 & 5.52 & 14 & 29.67 & 10.80 & 16 \\
& 10 & 34.47 & 4.68 & 4 & 28.70 & 59.47 & 36 & 32.44 & 11.36 & 18 & 32.44 & 20.03 & 17 \\
\midrule

\multirow{3}{*}{Iris} 
& 4  & 0.65 & 0.29 & 28 & 0.67 & 0.83 & 12 & 0.85 & 0.63 & 12 & 0.86 & 0.60 & 11 \\
& 7  & 0.65 & 0.39 & 29 & 0.68 & 1.74 & 15 & 0.55 & 1.44 & 20 & 0.66 & 0.92 & 10 \\
& 10 & 0.85 & 0.17 & 10 & 0.76 & 2.19 & 15 & 0.55 & 1.26 & 13 & 0.88 & 1.12 & 11 \\
\midrule

\multirow{3}{*}{Linnerud} 
& 4  & 35.08 & 0.30 & 8 & 28.71 & 0.50 & 17 & 35.08 & 0.59 & 18 & 35.08 & 0.96 & 20 \\
& 7  & 35.08 & 0.41 & 8 & 26.39 & 2.75 & 26 & 28.71 & 0.36 & 18 & 29.25 & 1.78 & 25 \\
& 10 & 29.24 & 0.63 & 9 & 22.50 & 3.10 & 28 & 29.24 & 1.50 & 26 & 35.09 & 2.01 & 24 \\
\midrule

\multirow{3}{*}{Wine} 
& 4  & 131.39 & 0.98 & 12 & 131.39 & 0.61 & 12 & 91.53 & 0.35 & 18 & 91.53 & 0.60 & 16 \\
& 7  & 131.39 & 1.22 & 12 & 75.71 & 1.48 & 19 & 131.39 & 0.46 & 10 & 131.39 & 0.77 & 9 \\
& 10 & 131.39 & 1.89 & 12 & 53.98 & 3.88 & 34 & 131.39 & 0.64 & 10 & 131.39 & 1.12 & 9 \\
\bottomrule
\end{tabular}%
}
\end{table}

%%%%%%%%%%%%%%%%%%%%%%%%%%
\begin{figure}[htbp]
    \centering
    \includegraphics[width=13cm]{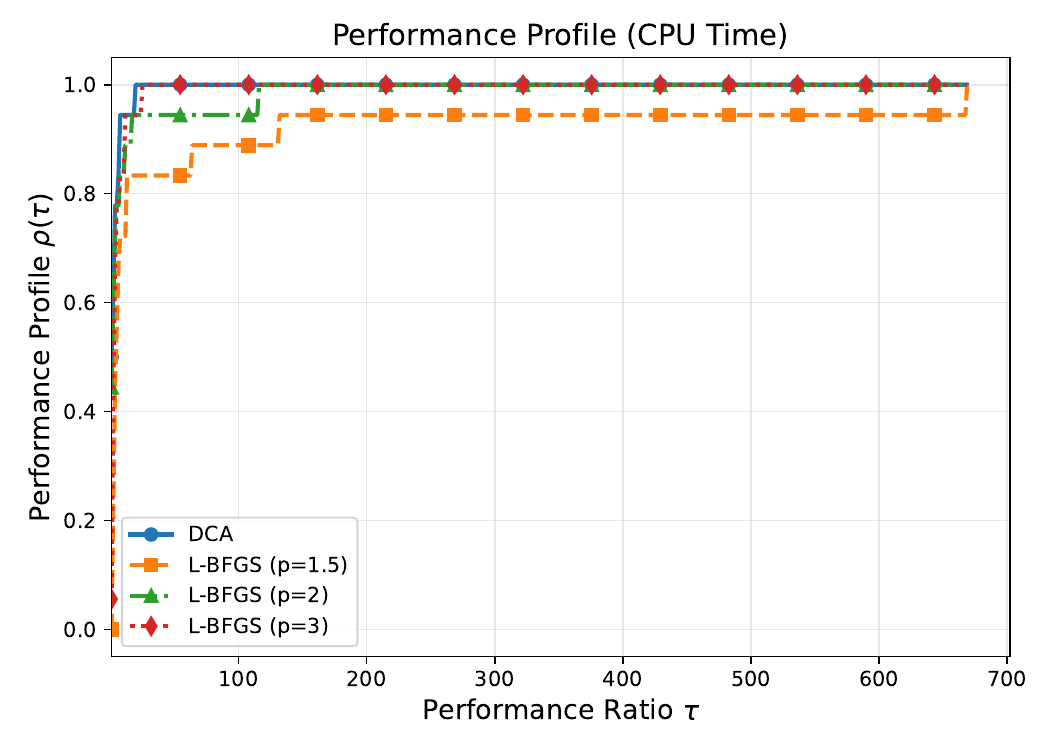}
    \vspace{-3mm}
    \caption{Performance profiles comparing DCA and IDEA with LBFGS directions in terms of CPU Time.}
    \label{fig:performance_time}
\end{figure}

%%%%%%%%%%%%%%%%%%%%%%%%%%%%%%%%%%%%%%%%%%%%%%%%%%%%%%%%%%%%%%%%%%%%%%%%%%%%%%%%%%%%%%%%%%%%%%%%%%%%%%%%%%%%%%%%%%%%%%%%%%%%%%%%%%%%%%%%%%%%%%%%%%%%
%%% Section: Concluding remarks
%%%%%%%%%%%%%%%%%%%%%%%%%%%%%%%%%%%%%%%%%%%%%%%%%%%%%%%%%%%%%%%%%%%%%%%%%%%%%%%%%%%%%%%%%%%%%%%%%
\section{Conclusion} \label{sec:Conclusion}
In this paper, we studied a class of descent methods (named IDEA) based on the difference of high-order Moreau envelopes (HOME-DC) for difference-of-convex (DC) optimization. We introduced HOME-DC and established its main analytical and differential properties, including approximation bounds, differential properties, and connections between DC-stationary points of the original DC problem and critical points of HOME-DC.
Next, we developed a generic IDEA on the basis of an inexact first-order oracle of HOME-DC, constructed using approximate proximal mappings. Under suitable assumptions on inexactness errors, we proved that every accumulation point of the generated sequence was a critical point of HOME-DC. Our preliminary numerical experiments on a sparse clustering problem confirmed the practical performance of the proposed approach.

%%%%%%%%%%%%%%%%%%%%%%%%%%%%%%%%%%%%%%%%%%%%%%%%%%%%%%%%%%%%%%%%%%%%%%%%%%%%%%%%%%%%%%%%%%%%%%%%%%%%%%%%%%%%%%%%%%%%%%%%%%%%%%%%%%%%%%%%%%%%%%%%%%%%
%%% End Section: Concluding remarks %%%%%%%%%%%%%%%%%%%%%%%%%%%%%%%%%%%%%%%%%%%%%%%%%%%%%%%%%%%%%%%%%%%%%%%%%%%%%%%%%%%%%%%%%%%%%%%%%

\subsection*{\textbf{Funding Information}}
The authors were partially supported by the Research Foundation Flanders (FWO) research project G081222N and UA BOF DocPRO4 projects with IDs 46929 and 48996.

\bibliographystyle{spbasic}
\bibliography{references}

\end{document}